\documentclass[11pt, a4paper]{article}
\usepackage{amsmath, amsfonts}
\usepackage[usenames]{color}
\usepackage[english]{babel}
\usepackage[table, dvipsnames]{xcolor}
\usepackage{array}
\newcolumntype{P}[1]{>{\centering\arraybackslash}p{#1}}
\newcolumntype{M}[1]{>{\centering\arraybackslash}m{#1}}
\usepackage{multirow}
\usepackage{tabularx}
\usepackage[all]{xy}
\usepackage{tikz-cd}
\usepackage{tikz}
\usepackage{float}
\usepackage{xcolor}
\usepackage{graphicx}
\usepackage{geometry}
\usepackage{pifont}
\usepackage{mathtools}
\usepackage{mathrsfs}
\usepackage{tcolorbox}
\usepackage[table]{xcolor}
\usepackage{hyperref} 
\hypersetup{bookmarks = true, 
                  colorlinks = true, 
                  linkcolor = blue, 
                  citecolor = blue, 
                  }
\usepackage{makecell}

\usepackage{tcolorbox}
\definecolor{mygray}{rgb}{0.98,0.95,0.75}
\newtcolorbox{mybox}{
    arc=0pt,
    boxrule=0pt,
    colback=mygray,
    width=1\textwidth,   
    colupper=black,
    fontupper=\bfseries
}

\newcommand{\xdownarrow}[1]{%
  {\left\downarrow\vbox to #1{}\right.\kern-\nulldelimiterspace}
}

\newcommand{\xuparrow}[1]{%
  {\left\uparrow\vbox to #1{}\right.\kern-\nulldelimiterspace}
}

\usepackage{graphicx}
\usepackage{amssymb}
\usepackage{amsmath}

\newtheorem{theorem}{Theorem}[section]

\newtheorem{conjecture}[theorem]{Conjecture}

\newtheorem{definition}[theorem]{Definition}
\newtheorem{example}[theorem]{Example}

\newtheorem{lemma}[theorem]{Lemma}

\newtheorem{proposition}[theorem]{Proposition}
\newtheorem{remark}[theorem]{Remark}

\newenvironment{proof}[1][Proof]{\noindent\textbf{#1.} }{\ \rule{0.5em}{0.5em}}

\title{\textbf{Equisingularity in families of double point curves}}
\author{ \ \ \ \\{Silva, O. N.\footnote{otoniel@ufu.br, Universidade Federal de Uberlândia, 38408-144, Uberlândia, MG, Brazil.} $ \ \ $ and  $ \ \ $ Silva Jr, M. M.\footnote{mmsj@academico.ufpb.br, Universidade Federal da Paraíba, 58.051-900, João Pessoa, PB, Brazil.} $ \ \ $}}
\date{}

\begin{document}

\maketitle

\begin{abstract}

In this paper, we provide a systematic comparison between the equisingularity of a 1-parameter unfolding $F = (f_t, t)$ of a finitely determined map germ $f: (\mathbb{C}^2, 0) \to (\mathbb{C}^3, 0)$ and the equisingularity of its associated families of double point curves: $D(F)$, $F(D(F))$, $D^2(F)$, and $D^2(F)/S_2$. We also construct explicit counterexamples to several natural questions concerning the equisingularity of these loci. As a key application, we introduce new families of complete intersection curves — referred to as Henry--type families — which are topologically trivial but fail to satisfy the Whitney equisingularity conditions. Finally, we generalize classical formulas for double point curves, originally established for map germs from $(\mathbb{C}^2, 0)$ to $(\mathbb{C}^3, 0)$, to the higher-dimensional setting of map germs from $(\mathbb{C}^n, 0)$ to $(\mathbb{C}^{2n-1}, 0)$ for $n \geq 3$, thereby endowing the associated curves with a convenient analytic structure.
\end{abstract}


\section{Introduction}

$ \ \ \ \ $ Let $f:(\mathbb{C}^2,0)\rightarrow (\mathbb{C}^3,0)$ be a finitely determined map germ and consider a $1$-parameter unfolding $F:(\mathbb{C}^{2} \times \mathbb{C},0)\rightarrow (\mathbb{C}^3 \times \mathbb{C},0)$ of $f$ given by $F(x,t)=(f_{t}(x),t)$. We assume that the origin is preserved, that is, $f_{t}(0)=0$ for all $t$. Several notions of equisingularity can be considered for $F$. In this work, we will deal only with topological triviality, Whitney equisingularity and bi-Lipschitz equisingularity of $F$ (see Section \ref{Unfoldingsection}). A characterization of the topological triviality of $F$ was given by Callejas Bedregal, Houston, and Ruas in \cite[Th. 6.2]{ruas-bedregal-houston} (see also \cite[Cor. 40]{bobadilla} and \cite[Th. 4.2]{juanjotomazella}). Some characterizations of Whitney equisingularity of $F$ have been given by Gaffney in \cite{gaffney}, Marar and Nuño-Ballesteros in \cite{marar-juanjo-guille} and more recently by the authors in \cite{zariskiotoniel} (see also \cite{slice} for the corank $1$ case). 

By Thom's second isotopy lemma for complex analytic maps (see \cite[Th. $5.2$]{ref5}), every Whitney equisingular unfolding $F$ of $f$ is also topologically trivial. However, the converse does not hold in general (see \cite[Ex. $5.5$]{ruas-otoniel}). In \cite[Th. 4.19]{zariskiotoniel}, the authors showed that if $f$ has corank $1$ and $F$ is a bi-Lipschitz equisingular unfolding of $f$, then $F$ is Whitney equisingular. Thus, a natural question arises: if $F$ is Whitney equisingular, is
it also bi-Lipschitz equisingular? The authors gave a negative answer to this question in the corank 2 case (see \cite[Prop. 4.18]{zariskiotoniel}). In this work, we present a negative answer to this question also in the corank 1 case (see Proposition \ref{lemma7}). Therefore, in our setting, for $f$ of any corank, we have the situation described in Figure \ref{diagram}.

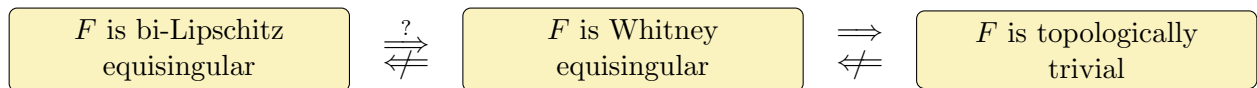
\begin{figure}[h!]
\centering    
\begin{tikzpicture}
[
  node distance=35mm and 45mm,
  box/.style={
    rounded corners=3pt, 
    draw, 
    align=center, 
    minimum width=45mm, 
    minimum height=10mm, 
    inner sep=4pt,
    fill=mygray
  },
  imply/.style={->, thick, shorten >=4pt, shorten <=4pt},
  equiv/.style={<->, thick, shorten >=4pt, shorten <=4pt},
  notimply/.style={->, thick, dashed, shorten >=4pt, shorten <=4pt},
  unknown/.style={->, thick, dotted, shorten >=4pt, shorten <=4pt},
  diagimply/.style={->, thick, shorten >=8pt, shorten <=8pt},
  diagnot/.style={->, thick, dashed, shorten >=8pt, shorten <=8pt},
  diagunk/.style={->, thick, dotted, shorten >=8pt, shorten <=8pt}
]

\node[box] (A) at (-6,0) {$F$ is bi-Lipschitz\\ equisingular};
\node[box] (B) at (0,0)   {$F$ is Whitney \\ equisingular };
\node[box] (C) at (6,0)    {$F$ is topologically \\ trivial };

\node (D) at (-3,0.2)    {$\overset{?}{\Longrightarrow}$};
\node(E) at (-3,-0.2) {$\Longleftarrow$};
\node(F) at (3,0.2) {$\Longrightarrow$};
\node (G) at (3,-0.2) {$\Longleftarrow$};
\node(H) at (-3,-0.2) {/};
\node(I) at (3,-0.2) {/};

\end{tikzpicture}
\caption{Implication diagram for Ruas’ conjecture.}
\label{diagram}
\end{figure}

In the study of the equisingularity of $F$, an important tool used by several authors is the family of double point curves $D^2(F)$ contained in $\mathbb{C}^2\times\mathbb{C}^2$ and its projection onto the first factor, namely, the family of plane curves $D(F)$ (see Section \ref{sec2.1}). For instance, Callejas Bedregal, Houston, and Ruas showed in \cite{ruas-bedregal-houston} (see also \cite[Cor. 40]{bobadilla}) that $F$ is topologically trivial if and only if the Milnor number of $D(f_t)$ is constant along the parameter space. Since the constancy of $\mu(D(f_t),0)$ is equivalent to the topological triviality of $D(F)$ (see, for instance, \cite[Th. 5.3.1]{buch}), the study of the topological triviality of the family of surfaces $f_t(\mathbb{C}^2)$ can be reduced to that of a family of plane curves in the source of $F$.

Other important families of curves in this context are the image of the double point curve family, denoted by $F(D(F))$, and the family $D^2(F)/S_2$ (see Section \ref{sec2.1} for details). Our main goal in this work is to clarify how the equisingularity properties of an unfolding $F$ control---and fail to control---the geometry of its associated double point locus. We analyze the families $D(F), F(D(F)), D^2(F),$ and $D^2(F)/S_2,$ obtaining explicit counterexamples. In this setting, a natural question arises:

\begin{mybox}
\textbf{Question 1:} Suppose that $F=(f_t,t)$ is equisingular in some sense. What can be said about the equisingularity of the families of curves $D(F)$, $F(D(F))$, $D^2(F),$ and $D^2(F)/S_2$?
\end{mybox}

For instance, if ``equisingular'' means ``topologically trivial'' then Callejas Bedregal, Houston, and Ruas showed in \cite{ruas-bedregal-houston} that all four families of double point curves $D(F)$, $D^2(F)$, $D^2(F)/S_2$ and $F(D(F))$ are topologically trivial. For the family $D^2(F)$, we can say more: we show that $D^2(F)$ is Whitney equisingular (see Lemma \ref{lemma1}).

Now, what happens if ``equisingular'' means ``Whitney equisingular'' (respectively, ``bi-Lipschitz equisingular'')? Is it true that $D(F)$, $F(D(F))$, $D^2(F),$ and $D^2(F)/S_2$ are also Whitney equisingular (respectively, bi-Lipschitz equisingular)? It is a well-known fact that if $\textbf{X}_t$ is a family of reduced plane curves, then these three notions of equisingularity are equivalent (see, for instance, \cite[Th. 5.2.2 and Th. 5.3.1]{buch} and \cite[Sec. VI]{pham2}). In particular, if $f$ is a finitely determined map germ from $(\mathbb{C}^2,0)$ to $(\mathbb{C}^3,0)$, then $D(f)$ is a reduced curve (see \cite[Cor. 3.5]{marar-juanjo-guille}; see also \cite[Cor. 3.5]{marar-mond}). In this case, if $F=(f_t,t)$ is an unfolding of $f,$ then these three notions of equisingularity are equivalent for the family $D(F),$ since topological triviality is equivalent to the constancy of $\mu(D(f_t),0)$ (see \cite[Th. 6.2]{ruas-bedregal-houston}). Thus, our study reduces to the three remaining families: $F(D(F))$, $D^2(F),$ and $D^2(F)/S_2$.

If $F$ is Whitney equisingular, then $F(D(F))$ and $D^2(F)$ are also Whitney equisingular (see Lemma \ref{lemma8}). However, we show that $D^2(F)/S_2$ may not be Whitney equisingular and $D^2(F)$ need not be bi-Lipschitz equisingular (see Proposition \ref{lemma2}). On the other hand, in some special cases we can ensure the bi-Lipschitz equisingularity of $D^2(F)$. For example, we show that if $f$ is homogeneous and $F$ is topologically trivial, then $D^2(F)$ is bi-Lipschitz equisingular (see Theorem \ref{lemma3}).

In \cite{thuy}, Brasselet, Ruas, and Thuy studied the bi-Lipschitz equivalence for pairs of map germs. In particular, it follows from \cite[Prop. 3.1]{thuy} that if $F$ is bi-Lipschitz equisingular, then $D^2(F)$ and $F(D(F))$ are also bi-Lipschitz equisingular. Moreover, in this work, we formulate a conjecture providing a criterion for determining whether an unfolding $F$ of a finitely determined map germ $f :(\mathbb{C}^2,0)\to(\mathbb{C}^3,0)$ is bi-Lipschitz equisingular (see Conjecture \ref{conjectura}). If this conjecture has a positive answer, it would be possible to construct an example showing that the bi-Lipschitz equisingularity of an unfolding $F$ does not imply that $D^2(F)/S_2$ is Whitney equisingular or bi-Lipschitz equisingular.

Among the four curves $D(f), D^2(f), f(D(f)),$ and $D^2(f)/S_2$, the curve $D^2(f)/S_2$ is certainly the least studied. However, it is important in that it can be used to construct families with interesting properties from the point of view of equisingularity. Henry showed that the family of curves $\textbf{X}_t\subset\mathbb{C}^3$ given by the quasihomogeneous complete intersection \begin{equation}\label{type}
    xy-tz=0\ \ \ \ \text{and}\ \ \ \ z^6+x^{15}+y^{10}=0
\end{equation} is topologically trivial but not Whitney equisingular (see, for instance, \cite[Ex. V.1]{Briancon}). Similarly to the classical examples of Briançon--Speder \cite{Speder}, which consist of $\mu$-constant families of singular surfaces in $\mathbb{C}^3$ that are topologically trivial but fail to satisfy the Whitney conditions, Henry’s family of curves has frequently appeared in the literature as a source of counterexamples to fundamental questions concerning equisingularity (cf. \cite{Briancon,buch,victor-juanjo}). Its ability to illustrate the subtle failure of the Whitney conditions in topologically trivial settings makes it a central example in the study of equisingularity theory. Inspired by Henry's example, we say that a family of curves $\textbf{X}_t\subset\mathbb{C}^3$ is of ``\textit{Henry--type}" if each $\textbf{X}_t$ is a quasihomogeneous complete intersection and the family $\textbf{X}_t$ is topologically trivial but not Whitney equisingular. In this sense, another natural question arises in this context:

\begin{mybox}
    \textbf{Question 2:} How can we construct new families of curves of Henry--type?
\end{mybox}

In this work, using the family of curves $D^2(F)/S_2$ associated with a suitable unfolding $F$, we provide an answer to Question 2 (see Propositions \ref{prop1} and \ref{prop2}).

In the second part of this work we generalize some classical formulas for double point curves, originally established for map germs from $(\mathbb{C}^2,0)$ to $(\mathbb{C}^3,0)$, to the higher dimensional setting of map germs from $(\mathbb{C}^n,0)$ to $(\mathbb{C}^{2n-1},0)$, with $n\ge 3.$ When $f$ is a finitely determined map germ from $(\mathbb{C}^2,0)$ to $(\mathbb{C}^3,0)$, Marar and Mond showed in \cite{marar-mond} (see also \cite{marar-juanjo-guille} for the case where $f$ has corank $2$) that
\begin{equation}\label{eq1}
\mu(D^2(f))=2\mu(D^2(f)/S_2)+C(f)-1,
\end{equation}

\noindent where $C(f)$ denotes the number of cross-caps that appear in a stabilization of $f$ near the origin. They also showed that
\begin{equation}\label{eq2}
\mu(D(f))=\mu(D^2(f))+6T(f)  \ \ \ \text{and} \ \ \ \mu(D(f))=2\mu(f(D(f)))+C(f)-2T(f)-1,
\end{equation}

\noindent where $T(f)$ denotes the number of triple points that appear in a stabilization of $f$ near the origin and all Milnor numbers in (\ref{eq1}) and (\ref{eq2}) are considered only at the origin.

In 2009, Jorge Pérez and Nuño-Ballesteros showed that the formula in (\ref{eq1}) also holds for finitely determined map germs from $(\mathbb{C}^n,0)$ to $(\mathbb{C}^{2n-1},0)$ with $n\geq 3$ (see \cite[Th. 1]{victor-juanjo}). They also considered a (not necessarily reduced) analytic structure on $D(f)$ and showed that the formula on the left side of (\ref{eq2}) holds only in some special cases (see \cite[Th. 2]{victor-juanjo}). Note that if $n \geq 3$, then $f$ has no ordinary triple points and therefore $T(f)=0$ in (\ref{eq2}). Thus, a natural question arises:

\begin{mybox}
\textbf{Question 3:} Let $f$ be a finitely determined map germ from $(\mathbb{C}^n,0)$ to $(\mathbb{C}^{2n-1},0)$, with $n\geq 3$. Can we endow the curves $D(f)$ and $f(D(f))$ with convenient analytic structures such that both formulas in (\ref{eq2}) hold?
\end{mybox}

We conclude this work by endowing the curves $D(f)$ and $f(D(f))$ with suitable analytic structures and providing a positive answer to Question $3$ (see Theorem \ref{TEO2}).

\section{Preliminaries}

$ \ \ \ \ $ In this work, unless otherwise stated, we assume that $f:(\mathbb{C}^2,0)\rightarrow(\mathbb{C}^{3},0)$ is a finite, generically $1$-to-$1$, holomorphic map germ. For any such finite map $f:\mathbb{C}^2\rightarrow\mathbb{C}^{3}$, we denote by $(x,y)$ and $(X,Y,Z)$ the coordinate systems on the source $\mathbb{C}^2$ and the target $\mathbb{C}^{3}$, respectively. We use throughout the paper the standard notation of singularity theory; see, for instance, \cite{mond-juanjo-livro}. We always consider the ambient metric, also referred to as the outer metric.

\subsection{Double point spaces for map germs}\label{sec2.1}

$ \ \ \ \ $ Multiple point spaces of a map germ from $(\mathbb{C}^n,0)$ to $(\mathbb{C}^p,0)$ with $n\leq p$ play an important role in the study of its geometry. In this section, we will deal only with double points, which are fundamental in the setting considered in this work.

We are interested in studying the space of double points, denoted by $D^2(f)$ and its projection onto the source of $f,$ denoted by $D(f)$. Roughly speaking, $D^2(f)$ is the set of points $(\textbf{x},\textbf{x}^{'}) \in \mathbb{C}^n \times \mathbb{C}^n$ such that $\textbf{x} \neq \textbf{x}^{'}$ and $f(\textbf{x})=f(\textbf{x}^{'})$ or $\textbf{x}$ is a singular point of $f$. In order to regard $D^2(f)$ as an analytic space, we need to present an appropriate analytic structure for it. We will follow the construction in \cite{mond87}, which is valid for holomorphic maps from $\mathbb{C}^n$ to $\mathbb{C}^p$, with $n\leq p$. 

Let us denote the diagonals of $\mathbb{C}^n \times \mathbb{C}^n$ and $\mathbb{C}^{p} \times \mathbb{C}^{p}$ by $\Delta_{\mathbb{C}^n}$ and $\Delta_{\mathbb{C}^{p}}$, respectively, and the sheaves of ideals defining them by $\mathcal{I}_n$ and $\mathcal{I}_{p}$, respectively. We write the points of $\mathbb{C}^n \times \mathbb{C}^n$ and $\mathbb{C}^{p} \times \mathbb{C}^{p}$ as $(\underline{x},\underline{x}^{'})$ and $(\underline{X},\underline{X}^{'})$, respectively.
Locally, 

\begin{center}
$\mathcal{I}_{n}=\langle x_1-x_1^{'},\cdots, x_n-x_n^{'} \rangle$\ \ \ and\ \ \  $\mathcal{I}_{p}=\langle X_1-X_1^{'},\cdots, X_{p}-X_{p}^{'} \rangle$. 
\end{center}

Since the pullback $(f \times f)^{\ast}\mathcal{I}_{p}$ is contained in $\mathcal{I}_{n}$ , there exist $\alpha_{ij}\in \mathcal{O}_{\mathbb{C}^{n} \times \mathbb{C}^n}$ such that
\[
f_{i}(\underline{x})-f_{i}(\underline{x}^{'})= \alpha_{i1}(\underline{x},\underline{x}^{'})(x_1-x_1^{'})+ \cdots + \alpha_{in}(\underline{x},\underline{x}^{'})(x_n-x^{'}_n), \ for \ i=1,\cdots,p.
\]

If $f(\underline{x})=f(\underline{x}^{'})$ and $\underline{x} \neq \underline{x}^{'}$, then every $n \times n$ minor of the matrix $\alpha=(\alpha_{ij})$ vanishes at $(\underline{x},\underline{x}^{'})$. We denote by $\mathcal{R}(\alpha)$ the ideal in $\mathcal{O}_{\mathbb{C}^{n}\times\mathbb{C}^n}$ generated by the $n\times n$ minors of $\alpha$. We then define the \textit{double point space $D^2(f)$} as a complex space by
\[
D^{2}(f)=V(\mathcal{I}^2(f)),
\] where $\mathcal{I}^2(f)=(f\times f)^{\ast}\mathcal{I}_{p}+\mathcal{R}(\alpha).$

Although the ideal $\mathcal{R}(\alpha)$ depends on the choice of coordinate functions of $f$, in \cite{mond87} it was proved that $\mathcal{I}^{2}(f)$ does not. Hence, $D^{2}(f)$ is well defined. It is easy to see that the points in the underlying set of $D^{2}(f)$ are exactly those of the form $(\underline{x},\underline{x}^{'})$ with $\underline{x} \neq \underline{x}^{'}$, $f(\underline{x})=f(\underline{x}^{'})$, together with those of the form $(\underline{x},\underline{x})$ such that $\underline{x}$ is a singular point of $f$.

Let $f:(\mathbb{C}^n,0)\rightarrow(\mathbb{C}^{p},0)$ be a finite map germ and take a representative of $f$ defined on a sufficiently small open neighborhood of the origin. Denote by $I_{p},$ $R(\alpha),$ and $I^2(f)$ the stalks at $0$ of $\mathcal{I}_{p},$ $\mathcal{R}(\alpha),$ and $\mathcal{I}^2(f),$ respectively. We define the \textit{double point space of the map germ $f$} as the complex space germ 

\[
(D^{2}(f),0)=(V(I^2(f)),0).
\]

We now specialize to the case $n=2$ and $p=3$. Other important spaces for studying the topology of $f(\mathbb{C}^{2})$ are the double point curve $D(f)$ in the source and its image under $f$, denoted by $f(D(f))$. According to the literature (see \cite{mond-pellikaan}), a suitable analytic structure for these curves is the one given by Fitting ideals. In what follows, $f_{\ast}\mathcal{O}_2$ denotes $\mathcal{O}_2$ considered as an $\mathcal{O}_3$-module, via composition with $f$, and $Fitt_k(f_{\ast}\mathcal{O}_2)$ is its $k$th Fitting ideal. The analytic structure of these curves is defined more precisely in the following definition:

\begin{definition}\label{DefD} Let $U \subset \mathbb{C}^2$ and $V \subset \mathbb{C}^{3}$ be open sets. Suppose that the map $f:U\rightarrow V$ is finite, that is, holomorphic, closed and finite-to-one. 

\begin{flushleft}
 (a) Let ${\pi}|_{D^2(f)}:D^2(f) \subset U \times U \rightarrow U$ be the restriction to $D^2(f)$ of the projection $\pi$ onto the first factor. The \textit{double point space} is the complex space
\end{flushleft}

\begin{center}
$D(f)=V(Fitt_0({\pi}_{\ast}\mathcal{O}_{D^2(f)}))$.
\end{center}

\noindent Set theoretically we have the equality $D(f)=\pi(D^{2}(f))$.\\

\noindent(b) The \textit{double point space in the target} is the complex space $f(D(f))=V(Fitt_1(f_{\ast}\mathcal{O}_2))$. Notice that the underlying set of $f(D(f))$ is the image of $D(f)$ by $f$. 

\end{definition}

\begin{remark}\label{remarkdimdf} (a) If $f:U \subset \mathbb{C}^2 \rightarrow V \subset \mathbb{C}^3 $ is finite and generically $1$-to-$1$, then $D^2(f)$ is Cohen-Macaulay and has dimension $1$ \rm\textit{(see} \rm\cite[\textit{Prop.} \rm 2.1]{marar-juanjo-guille}\textit{). Hence, $D^2(f)$, $D(f)$ and $f(D(f))$ are curves.}\\

\noindent\textit{(b) In our study, we will also consider the quotient double point space $D^2(f)/S_2$ of the double point space $D^2(f)$ by the action of the permutation group $S_2$.}

\textit{Assume that $G$ is a finite group acting linearly on $\mathbb{C}^n$. This action induces an analytic structure on the quotient $\mathbb{C}^n/G$ so that the local ring at a point $z\in\mathbb{C}^n$ is given by $$\mathcal{O}^G_{n,z}=\lbrace h\in\mathcal{O}_{n,z}:gh=h,\forall g\in G\rbrace.$$}

\textit{Assume now that $I\subset\mathcal{O}_{n,z}$ is a $G$-invariant ideal. Then $G$ acts also on the germ of analytic set $X=V(I)\subset(\mathbb{C}^n,z)$ and induces an analytic structure on $X/G$ with local ring $$\mathcal{O}_X^G=\lbrace h\in\mathcal{O}_X: gh=h,\forall g\in G\rbrace,$$ where $\mathcal{O}_X=\mathcal{O}_{n,z}/I$, in such a way that $X/G$ naturally embeds into $(\mathbb{C}^n/G,z).$ If $I$ is generated by $G$-invariant functions $a_1,\dots,a_r\in\mathcal{O}_{n,z},$ then $$\mathcal{O}_X^G\cong\mathcal{O}_{n,z}^G/I^G,$$ where $I^G$ is the ideal in $\mathcal{O}_{n,z}^G$ generated by the same functions $a_1,\dots,a_r.$ Since $\mathcal{O}_X^G$ is a subring of $\mathcal{O}_X,$ we have that if $X$ is reduced, then $X/G$ is also reduced.}

\textit{In our setting, if $f$ is a holomorphic map or a map germ from $\mathbb{C}^2$ to $\mathbb{C}^3$, then the double point ideal $\mathcal{I}^2(f)$ is $S_2$-invariant, where we consider the action
of the permutation group $S_2$ on $\mathbb{C}^2\times\mathbb{C}^2$ given by} \rm{$\tau(\textbf{x},\textbf{x}')=(\textbf{x}',\textbf{x}).$} \textit{In this way, we can define the quotient complex space or complex space germ $D^2(f)/S_2$. It is a well known fact that $\mathbb{C}^2\times\mathbb{C}^2/S_2$ is isomorphic to $\mathbb{C}^2$ times a quadratic cone in $\mathbb{C}^3$. In particular, $D^2(f)/S_2$ embeds into $\mathbb{C}^5$ (see \rm{\cite[\textit{Sec.} 3]{marar-mond}}}\textit{; see also} \rm{\cite[\textit{Sec.} 2.4]{teseguille}}\textit{).}\\

\noindent\textit{(c) We have the following commutative diagram.}

\begin{figure}[H]
    \centering
    \begin{tikzpicture}
        \draw (0,0) node {$D^2(f)$};
        
        \draw[->] (1,0) to (2.5,0);
        
        \draw (3.5,0) node {$D^2(f)/S_2$};
        
        \draw[->] (3.5,-0.5) to (3.5,-2);
        
        \draw (3.5,-2.5) node {$f(D(f))$};
        
        \draw[->] (1,-2.5) to (2.5,-2.5);
        
        \draw (0,-2.5) node {$D(f)$};
        
        \draw[->] (0,-0.5) to (0,-2);
    \end{tikzpicture}
\end{figure}

\noindent\textit{where the columns are generically $1$-to-$1$ and the lines are generically $2$-to-$1$. Furthermore, the definitions above extend to map germs from $(\mathbb{C}^n,0)$ to $(\mathbb{C}^{2n-1},0)$, with $n\geq3$, and we see that the columns in the diagram are generically $1$-to-$1$.}
\end{remark}

\subsection{Finite determinacy and quasihomogeneity}

$ \ \ \ \ $ The notion of a finitely determined map germ from $(\mathbb{C}^2,0)$ to $(\mathbb{C}^3,0)$ is central to this study, and we begin by recalling its classical definition. Furthermore, the number of cross-caps and triple points appearing in a stabilization of these germs plays a fundamental role throughout this paper.

\begin{definition}(a) Two map germs $f,g:(\mathbb{C}^2,0)\rightarrow (\mathbb{C}^3,0)$ are $\mathcal{A}$-equivalent, denoted by $g\sim_{\mathcal{A}}f$, if there exist germs of biholomorphisms $\Phi:(\mathbb{C}^2,0)\rightarrow (\mathbb{C}^2,0)$ and $\Psi:(\mathbb{C}^3,0)\rightarrow (\mathbb{C}^3,0)$, such that $g=\Psi \circ f \circ \Phi^{-1}$.\\

\noindent(b) A map germ $f:(\mathbb{C}^2,0) \rightarrow (\mathbb{C}^3,0)$ is $\mathcal{A}$-finitely determined (or simply finitely determined) if there exists a positive integer $k$ such that, for any map germ $g$ with $k$-jets satisfying $j^kg(0)=j^kf(0),$ we have $g \sim_{\mathcal{A}}f$.

\end{definition}

\begin{remark}\label{remarktriplepoints} Consider a finite map germ $f:(\mathbb{C}^2,0)\rightarrow (\mathbb{C}^3,0)$. By the Mather--Gaffney criterion \rm\cite[\textit{Th}. \rm 2.1]{Wall}\textit{, $f$ is finitely determined if and only if it has a finite representative $f:U \rightarrow V$, where $U\subset \mathbb{C}^2$ and $V \subset \mathbb{C}^3$ are open neighborhoods of the origin, such that $f^{-1}(0)=\lbrace 0 \rbrace$ and the restriction $f:U \setminus \lbrace 0 \rbrace \rightarrow V \setminus \lbrace 0 \rbrace$ is stable.} 

\textit{This means that the only singularities of $f$ on $U \setminus \lbrace 0 \rbrace$ are cross-caps (or Whitney umbrellas), transverse double points, and triple points. By shrinking $U$ if necessary, we may assume that there are no cross-caps or triple points in $U$. Since we are in the nice dimensions of Mather }\rm\cite[\textit{p}. \rm 208]{Mather}\textit{, we can then take a stabilization $F$ of $f$}, 

\begin{center}
$F:U \times T \rightarrow \mathbb{C}^4$, $F(x,y,t)=(f_{t}(x,y),t)$, 
\end{center}

\noindent \textit{where $T$ is a neighborhood of $0$ in $\mathbb{C}$. It is well known that the number of cross-caps of $f_t$ (denoted by $C(f)$) and the number of triple points of $f_t$ (denoted by $T(f)$) are independent of the particular choice of stabilization} \rm\textit{(see for instance} \rm\cite{mond7}\textit{).} \textit{These numbers are analytic invariants of $f$ (see} \rm\cite{mond87}\textit{).}

\end{remark}

We remark that the space $D(f)$ plays a fundamental role in the study of finite determinacy. In \cite[Th. 2.14]{marar-mond}, Marar and Mond presented necessary and sufficient conditions for a corank $1$ map germ from $(\mathbb{C}^n,0)$ to $(\mathbb{C}^p,0)$ to be finitely determined in terms of properties of $D^2(f)$ and other multiple point spaces. In the case $n=2$ and $p=3$, Marar, Nu\~{n}o-Ballesteros, and Pe\~{n}afort-Sanchis extended this criterion of finite determinacy to the corank $2$ case (see \cite[Cor. 3.5]{marar-juanjo-guille}). They proved the following result.

\begin{theorem}\rm(\cite{marar-juanjo-guille}, \cite{marar-mond})\label{criterio} \textit{
Let $f:(\mathbb{C}^2,0)\rightarrow(\mathbb{C}^{3},0)$ be a finite and generically $1$-to-$1$ map germ. Then $f$ is finitely determined if and only if the Milnor number of $D(f)$ at $0$ is finite.}
\end{theorem}

To conclude this section, let us recall the notion of quasihomogeneous map germs from $(\mathbb{C}^n,0)$ to $(\mathbb{C}^p,0)$ and germs of quasihomogeneous algebraic varieties in $(\mathbb{C}^n,0).$

\begin{definition}
    (a) A polynomial $p(x_1,\dots,x_n)$ is said to be quasihomogeneous if there exist positive integers $w_1,\dots,w_n$ with no common factor and a positive integer $d$ such that $$p(\lambda^{w_1}x_1,\dots,\lambda^{w_n}x_n)=\lambda^dp(x_1,\dots,x_n).$$ The integer $w_i$ is called the weight of the variable $x_i,$ and $d$  is called the weighted degree of $p.$ In this case, we say that $p$ is of type $(w_1,\dots,w_n;d).$\\
    
    \noindent(b) We say that $f:(\mathbb{C}^n,0)\rightarrow(\mathbb{C}^p,0)$  is a quasihomogeneous map germ of type $(w_1,\dots,w_n;d_1,\dots,d_p)$ if each coordinate function $f_i$ is quasihomogeneous of type $(w_1,\dots,w_n;d_i).$\\

    \noindent(c) Let $(V(I),0)\subset(\mathbb{C}^n,0)$ be a germ of algebraic variety, with $I=\langle f_1,\dots, f_r\rangle.$ We say that $(V(I),0)$ is quasihomogeneous of type $(w_1,\dots,w_n;d_1,\dots,d_r)$ if each $f_i$ is quasihomogeneous of type $(w_1,\dots,w_n;d_i).$
\end{definition}

\subsection{Classical notions of equisingularity in families of curves}\label{unfoldingsection}

$ \ \ \ \ $  Let us begin the section by reviewing the concepts of topological triviality, Whitney equisingularity and bi-Lipschitz equisingularity for families of curves, and by stating the main criteria that we will use. Consider a germ of a reduced, equidimensional complex surface $(X,0)$ in
$(\mathbb{C}^n,0)$ together with a flat analytic map
$p:(X,0)\to(\mathbb{C},0)$, and let $p:X\to T$ be a representative, where $T$ is an open neighborhood of $0$ in $\mathbb{C}$. The surface $X$ can be viewed as a flat $1$-parameter deformation of the curve $X_0:=p^{-1}(0)$. We assume that there exists a section $\sigma:T\to X$ such that $\sigma(T)$ is smooth and $X_t\setminus\{\sigma(t)\}$ is smooth for every $t\in T$ (see Figure \ref{Figure3}).

\begin{figure}[!h]
        \centering
\includegraphics[scale=0.5]{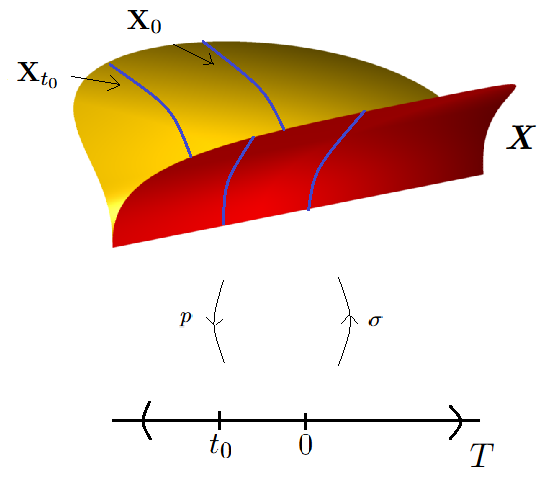} 
       \caption{The notion of a family of curves.}
        \label{Figure3}
    \end{figure}

    \begin{remark}
        The authors used \rm{\texttt{Surfer}} \textit{software} \rm{\cite{surfer}} \textit{to create Figure} \rm{\ref{Figure3}}.
    \end{remark}

    \begin{definition}\label{deftoptrivial} Let $p:(\textbf{X},0) \rightarrow (\mathbb{C},0)$ be a family of reduced curves and suppose that there is a good representative (in the sense of \cite{buch}) $p:\textbf{X}\rightarrow T$ with a section $\sigma:T \rightarrow \textbf{X}$ such that both $\sigma(T)$ and $\textbf{X}_t \setminus \sigma(t)$ are smooth for $t \in T$.\\ 

\noindent \textit{(a) We say that $p: \textbf{X}\rightarrow T$ is topologically trivial (or, for simplicity, that $\textbf{X}$ is topologically trivial) if there is a homeomorphism $h: \textbf{X} \rightarrow \textbf{X}_{0} \times T$ such that $p=p^{'} \circ h$, where $p^{'} :\textbf{X}_{0} \times T \rightarrow T$ is the projection onto the second factor}.\\

\noindent \textit{(b) Let $p:\textbf{X} \rightarrow T$ be a topologically trivial family of reduced curves. If, in addition, $h$ is bi-Lipschitz (that is, $h$ and its inverse $h^{-1}$ are Lipschitz maps with respect to the outer metric), then we say that $\textbf{X}$ is bi-Lipschitz equisingular.}\\

\noindent \textit{(c) We say that $p:\textbf{X}\rightarrow T$ is Whitney equisingular (or for simplicity, $\textbf{X}$ is Whitney equisingular) if the stratification} $ \lbrace \textbf{X} \setminus \sigma(T), \sigma(T) \rbrace $ \textit{satisfies  Whitney's condition (b) at $0$, that is:}\\

\noindent\textit{(Whitney's condition (b)): For any sequences of points $(x_n)\subset \textbf{X} \setminus \sigma(T)$ and $(t_n) \subset \sigma(T) \setminus \lbrace 0 \rbrace$, both converging to $0$, suppose that the sequence of lines $l_n$ joining $x_n$ and $t_n$ converges to a line $l$, and that the sequence of tangent spaces $T_{x_n}\textbf{X}$ converges to a linear space $H$. Then $l\subset H$.}       
    \end{definition}

A classical result asserts that $\textbf{X}$ is topologically trivial if and only if the Milnor number $\mu(\textbf{X}_t,\sigma(t))$ of its fibers remains constant along $\sigma(T)$ (see \cite[Th. $5.2.2$]{buch}). Furthermore, $\textbf{X}$ is Whitney equisingular if and only if it is topologically trivial and the multiplicity $m(\textbf{X}_t,\sigma(t))$ of $\textbf{X}_t$ at $\sigma(t)$ remains constant for all sufficiently small $t \in T$ (see \cite[Th. III.3]{Briancon}). Finally, $\textbf{X}$ is bi-Lipschitz equisingular if and only if the family of plane curves $p(\textbf{X}_t)$ is topologically trivial, where $p$ is a generic linear projection from $\mathbb{C}^n$ to $\mathbb{C}^2$ (see \cite[Ch. IV]{Briancon}; see also \cite[Cor. 3.6]{Giles}).\\

For convenience of the reader we recall the definition of the Milnor number of a space curve.  Let $(X,0)$ be a space curve germ in $(\mathbb{C}^{n},0)$ with isolated singularity. If $(X,0)$ is reduced, following \cite{buch} we define the \textit{Milnor number} as
\[
\mu(X,0)=\dim_{\mathbb{C}}\operatorname{coker}(\operatorname{cl}\circ d),
\]
where
\[
d:\mathcal{O}_{X}\longrightarrow \Omega_{X}
\]
is the differential operator, $\Omega_{X}$ is the module of holomorphic $1$-forms on $X$, $\operatorname{cl}:\Omega_{X}\to\omega_{X}$ is the class map and $\omega_{X}$ is the dualizing module of Grothendieck.
\begin{remark}
    Since \rm{\cite[\textit{Th.} $5.2.2$]{buch}} \textit{provides a criterion for determining when a family of curves is topologically trivial based on the constancy of the Milnor number, it is worth recalling that Milnor- and Tjurina-type numerical invariants serve as central tools in neighboring geometric contexts. In particular, in recent works, these invariants have played a vital role in the study of singular holomorphic foliations, where they capture the delicate interplay between the local geometry of the underlying foliation and the analytic properties of its separatrices (see} \rm{\cite{Arturo}}\textit{). Examining how Milnor and Tjurina numbers interact with geometric data in the setting of foliations not only provides criteria for equisingularity and topological triviality, but also highlights the broader versatility of Milnor-number methods in controlling complex analytic deformations.}
\end{remark}

The following result will be very useful for proving the flatness of some families of curves in Section \ref{Sec6}.

\begin{lemma}\label{lemaaux5}
Let $(X,0)\subset (\mathbb{C}^n,0)$ be a reduced complex analytic germ, pure of dimension $2$, and let $\pi:(X,0)\to(\mathbb{C},0)$ be the restriction of a linear projection. Set $X_0:=\pi^{-1}(0)$. If $\dim X_0=1$, then $\pi$ is flat.
\end{lemma}

\begin{proof}
Let $A=\mathcal{O}_{X,0}$ and let $t\in A$ be the function induced by $\pi$. It is enough to show that $t$ is not a zero divisor in $A$. Since $X$ is reduced, the zero divisors of $A$ form the union of its minimal prime ideals. If $t$ were a zero divisor, then $t\in\mathfrak p$ for some minimal prime $\mathfrak p$ of $A$. The corresponding irreducible component $X_i$ of $X$ would then satisfy $X_i\subset X_0$. Since $X$ is pure of dimension $2$, we have $\dim X_i=2$, and hence $\dim X_0\geq 2$, contradicting $\dim X_0=1$. Therefore, $t$ is not a zero divisor in $A$. Thus $A$ is torsion-free as a $\mathbb{C}\{t\}$-module. Since $\mathbb{C}\{t\}$ is a discrete valuation ring, every torsion-free module over it is flat. Hence $\pi$ is flat.
\end{proof}

\subsection{Equisingular unfoldings and families of double point curves}\label{Unfoldingsection}

$ \ \ \ \ $ Now, let us review the concepts of equisingular unfoldings of a finitely determined map germ from $(\mathbb{C}^2,0)$ to $(\mathbb{C}^3,0)$. Let $f:(\mathbb{C}^2,0)\rightarrow (\mathbb{C}^3,0)$ be a finitely determined map germ and consider a $1$-parameter unfolding $F:(\mathbb{C}^{2} \times \mathbb{C},0)\rightarrow (\mathbb{C}^3 \times \mathbb{C},0)$ of $f$ given by $F(x,t)=(f_{t}(x),t)$. We assume that the origin is preserved, that is, $f_{t}(0)=0$ for all $t$. We can consider several notions of equisingularity for $F$. In this work, we will deal only with topological triviality, Whitney equisingularity and bi-Lipschitz equisingularity of $F$.

Gaffney \cite{gaffney} introduced an important class of unfoldings called ``\textit{excellent unfoldings}''. An excellent unfolding has a natural stratification whose strata outside the parameter space $T$ (where $T$ is a small neighborhood of $0$ in $\mathbb{C}$ of a representative of $F$) are the stable types in source and target. For an unfolding $F$ as above, the strata in the source are as follows:

\begin{equation}\label{eq7}
\lbrace \mathbb{C}^2 \times \mathbb{C} \setminus D(F), \ D(F) \setminus T, \ T \rbrace.
\end{equation}

In the target, the strata are

\begin{equation}\label{eq8}
\lbrace \mathbb{C}^3 \times \mathbb{C} \setminus F(\mathbb{C}^2 \times \mathbb{C}), \  F(\mathbb{C}^2 \times \mathbb{C}) \setminus F(D(F)), \  F(D(F)) \setminus T, \ T \rbrace.
\end{equation}

\noindent where $D(F)$ denotes the double point space of the unfolding $F$, which can be viewed as a family of plane curves. Notice that $F$ preserves the stratifications, that is, $F$ maps each stratum into a stratum.

\begin{definition} Let $F:(\mathbb{C}^2 \times \mathbb{C},0)\rightarrow (\mathbb{C}^3 \times \mathbb{C},0)$ be a $1$-parameter unfolding of a finitely determined map germ $f:(\mathbb{C}^2,0)\rightarrow(\mathbb{C}^3,0)$.\\

\noindent (a) We say that $F$ is \textit{topologically trivial} if there exist homeomorphism germs
\[
\Phi :(\mathbb{C}^2 \times \mathbb{C},0)\rightarrow (\mathbb{C}^2 \times \mathbb{C},0),  \ \Phi(x,t)=(\phi_{t}(x),t), \ \phi_{0}(x)=x, \ \phi_{t}(0)=0 
\]
\[
\Psi:(\mathbb{C}^3 \times \mathbb{C},0)\rightarrow (\mathbb{C}^3 \times \mathbb{C},0), \ \Psi(x,t)=(\psi_{t}(x),t), \ \psi_{0}(x)=x, \ \psi_{t}(0)=0
\]
\noindent such that $I=\Psi^{-1} \circ F \circ \Phi$, where $I(x,t)=(f(x),t)$ is the trivial unfolding of $f$.\\

\noindent (b) If, in addition, the homeomorphisms $\Phi$ and $\Psi$ in (a) are bi-Lipschitz with respect to the outer metric, then we say that $F$ is \textit{bi-Lipschitz equisingular}.\\

\noindent (c) We say that $F$ is \textit{Whitney equisingular} if the stratifications \rm{\eqref{eq7}}
\textit{in the source and} \rm{\eqref{eq8}} \textit{in the target are Whitney regular along $T$.}  

\end{definition}

\begin{remark}\label{remark11}
   Let $f:(\mathbb{C}^2,0)\rightarrow (\mathbb{C}^3,0)$ be a finitely determined map germ and consider a $1$-parameter unfolding $F:(\mathbb{C}^{2} \times \mathbb{C},0)\rightarrow (\mathbb{C}^3 \times \mathbb{C},0)$ of $f$ given by $F(x,t)=(f_{t}(x),t)$.\\
   
   \noindent (a) The unfolding $F$ induces deformations of the (reduced) double point curves $D(f), f(D(f)), D^2(f),$ and $D^2(f)/S_2.$ These deformations give rise to the families of curves that we will denote by $D(F), F(D(F)), D^2(F),$ and $D^2(F)/S_2,$ respectively. Furthermore, given an unfolding $F$ of $f,$ we will consider $D^2(F)$ in $(\mathbb{C}^2\times\mathbb{C}^2\times\mathbb{C},0)$ rather than $(\mathbb{C}^3\times\mathbb{C}^3,0),$ and the other germs $D(F)\subset(\mathbb{C}^2\times\mathbb{C},0),\, F(D(F))\subset(\mathbb{C}^3\times\mathbb{C},0),$ and $D^2(F)/S_2\subset((\mathbb{C}^2\times\mathbb{C}^2\times\mathbb{C})/S_2,0).$\\
   
  \noindent (b) We use the following analytic structures:
\begin{enumerate}
\item $D^2(F)$ is endowed with the double point structure defined by
the Mond double point ideal $I^2(F)$;

\item $D(F)$ is endowed with the Fitting structure $D(F)=
V\!\left(
\operatorname{Fitt}_0
\big((p_1)_*\mathcal O_{D^2(F)}\big)
\right)$, where $p_1:D^2(F)\rightarrow \mathbb C^2\times T$ is the
projection onto the first source factor and the parameter;

\item $D^2(F)/S_2$ is endowed with the analytic quotient structure, that is, $\mathcal O_{D^2(F)/S_2}=\mathcal O_{D^2(F)}^{S_2};$

\item $F(D(F))$ is endowed with the Fitting structure $F(D(F))=
V\!\left(
\operatorname{Fitt}_1
\big(F_*\mathcal O_{\mathbb C^2\times T}\big)
\right).$
\end{enumerate}

\noindent (c) We note that the four associated families of curves $D(F)$, $D^2(F)$, $D^2(F)/S_2$, and $F(D(F))$ are flat deformations of $D(f)$, $D^2(f)$, $D^2(f)/S_2$, and $f(D(f))$, respectively; see \rm{\cite[\textit{Lemma} 4.2]{marar-juanjo-guille}} \textit{(see also} \cite[\textit{Prop.} 4.3]{ruas-bedregal-houston}\textit{). In particular, they are equidimensional (see} \cite[\textit{Lemma} 1.1]{Briancon}\textit{). Furthermore, the fibers $D(f_t)$, $D^2(f_t)$, $D^2(f_t)/S_2$ and $f_t(D(f_t))$ are reduced (see for instance} \cite[\textit{Prop.} 7.7]{Greuelnovo}\textit{).}\\

\noindent \textit{(d) In addition, suppose that $F$ is topologically trivial. Let $T$ be a sufficiently small representative of the parameter space. After shrinking $T$, if necessary, we can choose a good representative $p: D(F) \rightarrow T$, in the sense of} \cite{buch}\textit{, such that $p:D(F)\rightarrow T$ is topologically trivial.}

\textit{Since $p:D(F)\to T$ is a topologically trivial flat family of reduced plane curves, the Milnor number $\mu(D(f_t),0)$ is constant along $T$. Hence, by} \rm{\cite[\textit{Lemma} 5.2.1]{buch}}\textit{, the singularity does not split, that is, $D(f_t)\setminus\{0\}$
is smooth for every $t\in T$. In particular, the fibers of $D(F)$ are smooth away from the parameter axis. Moreover, if
$\mu(D(f_0),0)>0$, then $\mu(D(f_t),0)>0$ for every $t\in T$, and hence the origin is a singular point of $D(f_t)$ for every $t$.}

\textit{Similarly, by finite determinacy of $f$ and the properties of the
diagram in Remark} \rm{\ref{remarkdimdf}}\textit{, the fibers of the associated families $D^2(F)$, $D^2(F)/S_2$, and $F(D(F))$ are smooth away from their respective sections. Notice, however, that points of these sections need not be singular points of the corresponding total spaces.}
\end{remark}

 We say that $F$ is \textit{$\mu$-constant} if $\mu(D(f_{t}),0)$ is constant for all $t \in T$. In this setting, according to \cite[Cor. $32$]{bobadilla} (see also \cite[Th. $6.2$]{ruas-bedregal-houston}), the topological triviality of $F$ is characterized by the constancy of $\mu(D(f_t),0).$ In other words, $F$ is topologically trivial if and only if it is $\mu$-constant.

 Regarding the Whitney equisingularity of $F,$ Marar, Nuño-Ballesteros, and Peñafort-Sanchis \cite[Th. 5.3]{marar-juanjo-guille} proved that it is characterized by the simultaneous constancy of $\mu(\gamma_t,0)$ and $\mu(D(f_t),0).$ Here, $\gamma_t$ denotes the transverse slice of $f_t,$ which is obtained by intersecting $f_t(\mathbb{C}^2)$ with a generic plane in $\mathbb{C}^3$ (see \cite{slice}).

\section{Equisingularity in families of double point curves}

$ \ \ \ \ $ In this section we will provide an answer to Question $1$ in Introduction. Suppose that $f:(\mathbb{C}^2,0)\rightarrow (\mathbb{C}^3,0)$ is a finitely determined map germ and $F=(f_t,t)$ an unfolding of $f$. First, suppose that $F$ is topologically trivial. In this case, what can be said about the equisingularity of $D(F)$, $F(D(F))$, $D^2(F),$ and $D^2(F)/S_2$?

In \cite{ruas-bedregal-houston}, Callejas Bedregal, Houston, and Ruas showed that if $F$ is topologically trivial, then all four families of double point curves, namely $D(F)$, $D^2(F)$, $D^2(F)/S_2$ and $F(D(F))$ are topologically trivial. We present here an alternative proof of this result. Our contribution in Lemma \ref{lemma1} is limited to part (b); the other items are already known in the literature. 

\begin{lemma}\label{lemma1} Let $f:(\mathbb{C}^2,0)\rightarrow (\mathbb{C}^3,0)$ be a finitely determined map germ, and let $F=(f_t,t)$ be an unfolding of $f$. If $F$ is topologically trivial, then\\ 

\noindent (a) $D(F)$ is bi-Lipschitz equisingular.\\

\noindent (b) $D^2(F)$ is Whitney equisingular.\\

\noindent (c) $D^2(F)/S_2$ and $F(D(F))$ are topologically trivial.

\end{lemma}

\begin{proof} (a) Suppose that $F$ is topologically trivial. Since the homeomorphisms in the source of $F$ must map the double point curve $D(f_t)$ onto $D(f)$, it follows that $D(F)$ is a topologically trivial deformation of $D(f)$. Since $D(F)$ is a family of plane curves, it is also bi-Lipschitz equisingular (see \cite[Th. 1.1]{Pichon}).\\

\noindent (c) For each $t$, it follows from \cite{marar-mond} and \cite{marar-juanjo-guille} that  
\begin{equation}\label{eq6}
\mu(D(f_t))=\mu(D^2(f_t))+6T(f_t)  \ \ \ and \ \ \ \mu(D(f_t))=2\mu(f_t(D(f_t)))+C(f_t)-2T(f_t)-1.
\end{equation}

Since $C$ and $T$ are topological invariants (see \cite{edson3}) and $F$ is topologically trivial, the upper semicontinuity of the invariants appearing in (\ref{eq6}) implies that $\mu(D^2(f_t))$ and $\mu(f_t(D(f_t)))$ are constant. Therefore, $D^2(F)$ and $F(D(F))$ are topologically trivial. Moreover, since
\begin{equation}
\mu(D^2(f_t))=2\mu(D^2(f_t)/S_2)+C(f_t)-1,
\end{equation}

\noindent it follows that $\mu(D^2(f_t)/S_2)$ is also constant. Therefore, $D^2(F)/S_2$ is also topologically trivial.\\

\noindent(b) From the proof of (c), it follows that $D^2(F)$ is topologically
trivial. Therefore, it suffices to show that the multiplicity of
$D^2(f_t)$ is constant along $T$ (see \cite[Th. III.3]{Briancon}).
After shrinking the parameter space if necessary, the branches of
$D^2(f_t)$ can be labelled consistently with $t$. Since the
multiplicity of a reduced curve is the sum of the multiplicities of
its irreducible branches, it is enough to show that the multiplicity
of each branch of $D^2(f_t)$ is constant along $T$.

Let $D(F)^j$ be an irreducible component of $D(F)$. Since $D(F)$ is Whitney equisingular, the multiplicity $m(D(f_t)^j)$ is constant along $T$, namely, $m_j$. By \cite[Th. 5.2.2]{buch}, $p:D(F) \rightarrow T$ admits a simultaneous resolution. After shrinking the parameter space if necessary, the simultaneous resolution allows us to label the branches of \(D(f_t)\) consistently with \(t\). Since the \(S_2\)-involution permutes this finite set of branches and the induced permutation is locally constant in \(t\), the branch paired with \(D(f_t)^j\) is a fixed branch \(D(f_t)^l\) for all sufficiently small \(t\).
 Furthermore, there is a parametrization in family $\varphi^j$ of $D(f_t)^j$ given by 
\begin{equation*}
\varphi^j(u):= \left( u^{m_j}, \displaystyle { \sum_{i\geq m_j}^{}}a_{j,i}u^i  \right).
\end{equation*}
\noindent where the $a_{j,i}$ are in $\mathbb{C}\lbrace t \rbrace$.

Suppose that $D^2(F)^j$ is an irreducible component of $D^2(F)$ such that $p(D^2(F)^j)=D(F)^j$, where $p:\mathbb{C}^2 \times \mathbb{C}^2 \times T\rightarrow \mathbb{C}^2 \times T$ is the projection $p(x,y,x',y',t)=(x,y,t)$. Since $D^2(f_t)$ is $S_2$-invariant,
there exists an irreducible component $D(F)^l$ of $D(F)$, of constant multiplicity $m_l$ (not necessarily distinct from $D(F)^j),$ which also admits a parametrization in family $\varphi^l_t$ given by
\begin{equation*}
\varphi^l(u):= \left( u^{m_l}, \displaystyle { \sum_{i\geq m_l}^{}}b_{l,i}u^i  \right).
\end{equation*}
\noindent where $b_{l,i} \in  \mathbb{C}\lbrace t \rbrace$.  It follows from \cite[Lemma 3.2]{greuel-lossen} that there exists an invertible element $\alpha\in\mathcal{O}_1$ and a parametrization in family $\tilde{\varphi}^j$ of $D^2(F)^j$ such that 
\begin{equation*}
\tilde{\varphi}^j(u):= \left( u^{m_j}, \displaystyle { \sum_{i\geq m_j}^{}}a_{j,i}u^i, (\alpha u)^{m_l}, \displaystyle { \sum_{i\geq m_l}^{}}b_{l,i}(\alpha u)^i   \right).
\end{equation*}

Therefore, the multiplicity of $D^2(f_t)^j$ is the minimum of $m_j$ and $m_l$. 

Since the multiplicity is additive, it follows that the multiplicity of $D^2(f_t)$ is constant along $T$ and therefore $D^2(F)$ is Whitney equisingular.\end{proof}\\

The next lemma will be useful for computing explicit generators of $D^2(f)/S_2,$ where $f$ is a finitely determined corank 1 map germ from $(\mathbb{C}^2,0)$ to $(\mathbb{C}^3,0).$

\begin{lemma}\label{poly}
    For each integer $k\ge1$, let $P_k$ be the symmetric polynomial given by \begin{equation}\label{poly1}
        P_k(y,z)=\sum_{i=0}^{k}y^{k-i}z^{i}.
    \end{equation}
    
    Then, $P_k$ can be written in the form \begin{equation}\label{poly2}
        P_k(y,z)=\dfrac{1}{2^k}\sum_{i=0}^{\lfloor k/2\rfloor}\binom{k+1}{2i+1}(y+z)^{k-2i}((y-z)^2)^i.
    \end{equation}
\end{lemma}

\begin{proof}
    First, let us consider the case $y=z.$ By \eqref{poly1}, we have $$P_k(y,y)=\sum_{i=0}^{k}y^k=(k+1)y^k.$$
    
    On the right-hand side of \eqref{poly2}, we have $$\dfrac{1}{2^k}\sum_{i=0}^{\lfloor k/2\rfloor}\binom{k+1}{2i+1}(y+y)^{k-2i}((y-y)^2)^i=\dfrac{1}{2^k}\binom{k+1}{1}(2y)^k=(k+1)y^k,$$ which proves the equality in \eqref{poly2} in this case.
    
    Now, suppose that $y\ne z.$ Setting $s=y+z$ and $d=y-z,$ note that $y=(s+d)/2$ and $z=(s-d)/2.$ On the other hand, by \eqref{poly1}, observe that $P_k(y,z)=(y^{k+1}-z^{k+1})/(y-z).$ Therefore,  \begin{eqnarray*}
        P_k(y,z)&=&\dfrac{1}{d}\left(\left(\dfrac{s+d}{2}\right)^{k+1}-\left(\dfrac{s-d}{2}\right)^{k+1}\right)\\
        &=&\dfrac{1}{d}\left(\sum_{i=0}^{k+1}\binom{k+1}{i}\left(\dfrac{s}{2}\right)^{k+1-i}\left(\dfrac{d}{2}\right)^{i}-\sum_{i=0}^{k+1}\binom{k+1}{i}\left(\dfrac{s}{2}\right)^{k+1-i}\left(\dfrac{-d}{2}\right)^{i}\right)\\
        &=&\dfrac{1}{2^{k+1}d}\sum_{i=0}^{k+1}\binom{k+1}{i}s^{k+1-i}(d^i-(-d)^i).
    \end{eqnarray*}
    
    Since $d^i-(-d)^i=0,$ for all even integers $i$, only the terms corresponding to odd indices appear in the sum above. Thus, the sum has only $\lfloor k/2\rfloor+1$ terms. Hence, \begin{eqnarray*}
        P_k(y,z)&=&\dfrac{1}{2^{k+1}d}\sum_{i=0}^{\lfloor k/2\rfloor}\binom{k+1}{2i+1}s^{k+1-(2i+1)}(2d^{2i+1})\\
        &=&\dfrac{1}{2^k}\sum_{i=0}^{\lfloor k/2\rfloor}\binom{k+1}{2i+1}s^{k-2i}(d^2)^i,
    \end{eqnarray*} which proves the equality in \eqref{poly2}.
\end{proof}\\

Now, we can use Lemma \ref{poly} to obtain explicit generators of $D^2(f)/S_2$ for a finitely determined corank 1 map germ $f$ from $(\mathbb{C}^2,0)$ to $(\mathbb{C}^3,0)$ considered below. In what follows, whenever $f$ is a finitely determined corank 1 map
germ, we replace it, if necessary, by a polynomial representative in
its $\mathcal{A}$-equivalence class. This is always possible. Indeed, since $f$ has corank 1, by a holomorphic source and target coordinate changes, we can write $f$ in the form \[
f(x,y)=(x,p(x,y),q(x,y))
\]
for some $p,q\in\mathfrak{m}_2^2$. In particular, since $f$ is finite, the restrictions $p(0,y)$ and $q(0,y)$ cannot both vanish identically. Indeed, otherwise $f(0,y)=(0,0,0)$ for every $y$, contradicting the finiteness of $f$. Finally, since $f$ is $k$-determined, we have $f\sim_{\mathcal A}j^k f$, and $j^k f$ is polynomial. Therefore, after a suitable change of coordinates
we may assume that
 $p$ and $q$ are polynomials whose terms have degree at least $2$. Hence, we may write
\begin{equation}\label{p}
    p(x,y)=a_1(x)y+\cdots+a_m(x)y^m
\end{equation}
and
\begin{equation}\label{q}
    q(x,y)=b_1(x)y+\cdots+b_n(x)y^n.
\end{equation}
Furthermore, there exist indices $i$ and $j$ and constants $\alpha$ and $\beta$ such that $\alpha:=a_i(0)$ and $\beta:=b_j(0)$ are not both zero. In several applications considered below, we deal with quasihomogeneous
map germs for which the highest powers of $y$ appearing in $p$ and $q$
are pure powers. In those cases, we shall assume that $a_m(x)\equiv \alpha$ and $b_n(x)\equiv \beta$, that is, \begin{eqnarray*}
    f(x,y)=(x, \ a_1(x)y+\dots+a_{m-1}(x)y^{m-1}+\alpha y^m, \ b_1(x)y+\dots+b_{n-1}(x)y^{n-1}+\beta y^n).
\end{eqnarray*} In what follows, we present explicit generators of $D^2(f)/S_2$ for a finitely determined corank 1 map germ $f$ from $(\mathbb{C}^2,0)$ to $(\mathbb{C}^3,0).$

\begin{proposition}\label{geradores}
    Let $f:(\mathbb{C}^2,0)\rightarrow(\mathbb{C}^3,0)$ be a finitely determined corank 1 map germ. Write $f$ in the normal form $f(x,y)=(x,p(x,y),q(x,y))$ for some $p,q\in\mathfrak{m}_2^2$ as in \eqref{p} and \eqref{q}. Then, $$D^2(f)/S_2=V(\overline{\Gamma}_1(x,\sigma_1,\sigma_2),\overline{\Gamma}_2(x,\sigma_1,\sigma_2)),$$ where $$\sigma_1=y+z, \ \ \ \sigma_2=(y-z)^2,\ \ \ \overline{P}_k(\sigma_1,\sigma_2)=\dfrac{1}{2^k}\sum_{i=0}^{\lfloor k/2\rfloor}\binom{k+1}{2i+1}\sigma_1^{k-2i}\sigma_2^i,$$ for each integer $k\ge1,$ $$\overline{\Gamma}_1(x,\sigma_1,\sigma_2)=a_1(x)+a_2(x)\overline{P}_1(\sigma_1,\sigma_2)+\dots+a_{m-1}(x)\overline{P}_{m-2}(\sigma_1,\sigma_2)+a_{m}(x) \overline{P}_{m-1}(\sigma_1,\sigma_2),$$ and $$\overline{\Gamma}_2(x,\sigma_1,\sigma_2)=b_1(x)+b_2(x)\overline{P}_1(\sigma_1,\sigma_2)+\dots+b_{n-1}(x)\overline{P}_{n-2}(\sigma_1,\sigma_2)+b_n(x) \overline{P}_{n-1}(\sigma_1,\sigma_2).$$
\end{proposition}

\begin{proof}
    Using the technique of divided differences, one can compute explicit generators of $D^2(f).$ More precisely, $D^2(f)$ is given by \begin{eqnarray*}
        D^2(f)&=&V\left(\dfrac{p(x,y)-p(x,z)}{y-z},\dfrac{q(x,y)-q(x,z)}{y-z}\right)\\
        &=&V(\Gamma_1(x,y,z),\Gamma_2(x,y,z)),
    \end{eqnarray*} where $$\Gamma_1(x,y,z)=a_1(x)+a_2(x)P_1(y,z)+\dots+a_{m-1}(x)P_{m-2}(y,z)+a_m(x) P_{m-1}(y,z)$$ and $$\Gamma_2(x,y,z)=b_1(x)+b_2(x)P_1(y,z)+\dots+b_{n-1}(x)P_{n-2}(y,z)+b_n(x) P_{n-1}(y,z),$$ with $P_k(y,z)$ denoting the symmetric polynomial defined in Lemma \ref{poly}.

    On the other hand, in \cite[Sec. 2.4]{teseguille}, a method is given for obtaining generators of $D^2(f)/S_2$ for a map germ $f$ from $(\mathbb{C}^2,0)$ to $(\mathbb{C}^3,0)$ in the case of corank 1. First, consider the map $\psi:\mathbb{C}^3\rightarrow\mathbb{C}^3$ given by $$\psi(x,y,z)=(x,y+z,(y-z)^2).$$
    
    This map is such that $\psi(D^2(f))\simeq D^2(f)/S_2$ locally. Thus, $D^2(f)/S_2$ is generated by the ideal $(\psi^*)^{-1}(\mathcal{I}^2(f)^{S_2}),$ where $\mathcal{I}^2(f)$ is the ideal defining $D^2(f).$ In particular, if we write the generators of $D^2(f)$ as polynomials in the variables $x,\sigma_1,$ and $\sigma_2,$ then we obtain explicit generators of $D^2(f)/S_2.$

    By Lemma \ref{poly}, we have $P_k(y,z)=(1/2^k)\sum_{i=0}^{\lfloor k/2\rfloor}\binom{k+1}{2i+1}(y+z)^{k-2i}((y-z)^2)^i.$ Changing to the variables $\sigma_1$ and $\sigma_2,$ we can write $\Gamma_1$ and $\Gamma_2$ in terms of the variables $x,\sigma_1,$ and $\sigma_2.$ Denote by $\overline{\Gamma}_1(x,\sigma_1,\sigma_2)$ and $\overline{\Gamma}_2(x,\sigma_1,\sigma_2)$ the polynomials $\Gamma_1$ and $\Gamma_2$, respectively, after this change of variables. We then obtain $D^2(f)/S_2=V(\overline{\Gamma}_1(x,\sigma_1,\sigma_2),\overline{\Gamma}_2(x,\sigma_1,\sigma_2)),$ where $\overline{P}_k(\sigma_1,\sigma_2)=(1/2^k)\sum_{i=0}^{\lfloor k/2\rfloor}\binom{k+1}{2i+1}\sigma_1^{k-2i}\sigma_2^i,$ for each integer $k\ge1,$ and $$\overline{\Gamma}_1(x,\sigma_1,\sigma_2)=a_1(x)+a_2(x)\overline{P}_1(\sigma_1,\sigma_2)+\dots+a_{m-1}(x)\overline{P}_{m-2}(\sigma_1,\sigma_2)+a_m(x)
    \overline{P}_{m-1}(\sigma_1,\sigma_2),$$ and $$\overline{\Gamma}_2(x,\sigma_1,\sigma_2)=b_1(x)+b_2(x)\overline{P}_1(\sigma_1,\sigma_2)+\dots+b_{n-1}(x)\overline{P}_{n-2}(\sigma_1,\sigma_2)+b_n(x) \overline{P}_{n-1}(\sigma_1,\sigma_2),$$ which concludes the proof.
\end{proof}\\

Ruas and the first author showed in \cite{ruas-otoniel} that if $F$ is topologically trivial, then $F(D(F))$ need not be Whitney equisingular. In the next proposition, we show that, even if $F$ is topologically trivial, $D^2(F)$ need not be bi-Lipschitz equisingular, while $D^2(F)/S_2$ need not be Whitney equisingular or bi-Lipschitz equisingular. Therefore, the conclusions about the equisingularity of the families of curves $D(F)$, $D^2(F)$, $D^2(F)/S_2,$ and $F(D(F))$ in Lemma \ref{lemma1}, under the hypothesis that $F$ is topologically trivial, cannot be improved.

Before presenting the proposition, let us recall the algebraic version of the Whitney tangent cone (or simply tangent cone) $C_3(\mathcal{W},0)=C(\mathcal{W},0)$ of an analytic set $\mathcal{W}\subset\mathbb{C}^n$ at the origin, following Whitney's construction in \cite[Sec. 5]{Whi652}.

\begin{definition}\label{defc3}
    (a) For an analytic function $f$ at $0\in\mathbb{C}^n,$ we can write $f=f_m+f_{m+1}+\dots,$ where $f_i$ is a homogeneous polynomial of degree $i$ for each $i.$ If $f_m\ne0,$  then $f_m$ is called the initial form of $f,$ denoted by $f^{[*]}.$\\

    \noindent(b) Let $\mathcal{W}$ be an analytic set in $\mathbb{C}^n,$ with $0\in\mathcal{W},$ and let $\mathcal{I}_{(\mathcal{W},0)}$ be the ideal defining the germ $(\mathcal{W},0)$ in $(\mathbb{C}^n,0).$ The tangent cone $C(\mathcal{W},0)$ of $\mathcal{W}$ at $0$ is the algebraic set defined by the ideal generated by the initial forms of all $f\in\mathcal{I}_{(\mathcal{W},0)}.$ In particular, if
$\mathcal{I}_{(\mathcal{W},0)}=\langle f_1,\dots,f_r\rangle,$
where $f_1,\dots,f_r$ form a standard basis of
$\mathcal{I}_{(\mathcal{W},0)},$ then $C(\mathcal{W},0)=V\bigl(f_1^{[*]},\dots,f_r^{[*]}\bigr)$, see \rm{\cite[\textit{Chapter} 2]{Chirka}}.
\end{definition}

\begin{proposition}\label{lemma2} Consider the family of map germs $F=(f_t,t),$ where $f_t$ is defined by
\begin{equation}\label{eq4}
f_t(x,y)=(x,y^4+txy,y^7+x^2y).
\end{equation}

For this family, $f_0=f$ is finitely determined, and $F=(f_t,t)$ is Whitney equisingular (and, in particular, it is topologically trivial). Furthermore, the family of double point curves $D^2(F)$ is not bi-Lipschitz equisingular, while $D^2(F)/S_2$ is neither Whitney equisingular nor bi-Lipschitz equisingular.

\end{proposition}

\begin{proof} Since $f$ is a corank 1 map germ, we can apply Marar--Mond criterion \cite[Th. 2.14]{marar-mond} to verify the finite determinacy of $f.$ Using the following \texttt{Singular} code \cite{singular}
\begin{flushleft}
    \texttt{LIB "sing.lib";} 
    
    \texttt{ring r=0,(x,y,z,w),ds;}
    
    \texttt{poly p=y4;}
    
    \texttt{poly q=y7+x2y;}
    
    \texttt{ideal D2=(p-subst(p,y,z))/(z-y),(q-subst(q,y,z))/(z-y),w;}\ \ \ // \textit{the ideal defining $D^2(f)$}
    
    \texttt{string s2="The dimension of the double point space $D^2(f)$ is equal to:";}
    
    \texttt{s2;}
    
    \texttt{dim(std(D2));}
    
    \texttt{ideal sD2=slocus(D2);}\ \ \ // \textit{the ideal defining $Sing(D^2(f))$}
    
    \texttt{string sl2="The dimension of the singular locus of the double point space $D^2(f)$ is equal to:";}
    
    \texttt{sl2;}
    
    \texttt{dim(std(sD2));}

    \texttt{poly h1=(D2[1]-subst(D2[1],z,w))/(w-z);}

    \texttt{poly h2=(D2[2]-subst(D2[2],z,w))/(w-z);}
    
    \texttt{ideal D3=D2[1],D2[2],h1,h2;}\ \ \ // \textit{the ideal defining $D^3(f)$}
    
    \texttt{string s3="The dimension of the triple point space $D^3(f)$ is equal to:";}
    
    \texttt{s3;}
    
    \texttt{dim(std(D3));}

    \texttt{//->\,The dimension of the double point space $D^2(f)$ is equal to:}

    \texttt{//->\,1}

    \texttt{//->\,The dimension of the singular locus of the double point space $D^2(f)$ is equal to:}

    \texttt{//->\,0}

    \texttt{//->\,The dimension of the triple point space $D^3(f)$ is equal to:}

    \texttt{//->\,0.}
\end{flushleft} we conclude that $D^2(f)$ is an ICIS of dimension 1, and that $D^3(f)$ is an ICIS of dimension 0. Hence, by Mond--Marar criterion \cite[Th. 2.14]{marar-mond}, it follows that $f$ is finitely determined.

Note that $f$ is quasihomogeneous of type $(3,1;3,4,7)$. Since the unfolding adds only a term of weighted degree $4$ to the second coordinate function, whose weighted degree is also $4$, it follows from \cite[Cor. 4.4]{Slice} that $F$ is Whitney equisingular. Using divided differences, we obtain \begin{eqnarray*}
    D^2(F)&=&V\left(\dfrac{y^4+txy-z^4-txz}{y-z},\,\dfrac{y^7+x^2y-z^7-x^2z}{y-z}\right)\\
    &=&V(y^3+y^2z+yz^2+z^3+tx,\  x^2+y^6+y^5z+y^4z^2+y^3z^3+y^2z^4+yz^5+z^6).
\end{eqnarray*}

For $t\ne0$ sufficiently small, consider the isomorphism given by $x\mapsto -(y^3+y^2z+yz^2+z^3)/t.$ Therefore, $D^2(f_t)$ is isomorphic to a plane curve given by $C_t:=V((1+t^2)(y^6+z^6)+(2+t^2)(y^5z+yz^5)+(3+t^2)(y^4z^2+y^2z^4)+(4+t^2)y^3z^3)\subset\mathbb{C}^2$ whose tangent cone is $C_t$ itself. In particular, since \(f_t\) is finitely determined, \(D^2(f_t)\) is reduced, and hence the plane curve \(V(C_t)\) is reduced as well.
 As the defining equation of \(C_t\) is homogeneous of degree \(6\), it factors over \(\mathbb{C}\) as a product of six distinct linear forms. Hence \(D^2(f_t)\) has six distinct tangent lines. On the other hand, for $t=0$, this is no longer the case. In fact, since the generators $y^3+y^2z+yz^2+z^3$ and $x^2+y^6+y^5z+y^4z^2+y^3z^3+y^2z^4+yz^5+z^6$ form a standard basis for $D^2(f)$, then  Definition \ref{defc3} yields $$C(D^2(f))=V(y^3+y^2z+yz^2+z^3,\  x^2),$$ which implies that the tangent lines are given by the 3 distinct lines determined by $x=0$ and $y^3+y^2z+yz^2+z^3=0.$ Therefore, $D^2(F)$ is not bi-Lipschitz equisingular (see \cite[Th. 2.2]{Sampaio}; see also \cite[Cor. 4.7]{Giles}). Finally, let us show that $D^2(F)/S_2$ is not Whitney equisingular.

By Proposition \ref{geradores}, we obtain \begin{eqnarray*}
    D^2(F)/S_2&=&V\left(\overline{P}_3(\sigma_1,\sigma_2)+tx,\overline{P}_6(\sigma_1,\sigma_2)+x^2\right)\\
    &=&V\left(\dfrac{1}{2^3}\sum_{i=0}^{1}\binom{4}{2i+1}\sigma_1^{3-2i}\sigma_2^i+tx,\,\dfrac{1}{2^6}\sum_{i=0}^{3}\binom{7}{2i+1}\sigma_1^{6-2i}\sigma_2^i+x^2\right)\\
    &=&V\left(\dfrac{1}{8}(4\sigma_1^3+4\sigma_1\sigma_2)+tx,\,\dfrac{1}{64}(7\sigma_1^6+35\sigma_1^4\sigma_2+21\sigma_1^2\sigma_2^2+\sigma_2^3)+x^2\right)\\
    &=&V(\sigma_1^3+\sigma_1\sigma_2+2tx,\,7\sigma_1^6+35\sigma_1^4\sigma_2+21\sigma_1^2\sigma_2^2+\sigma_2^3+64x^2).
\end{eqnarray*}

For $t=0,$ we use the following \texttt{Singular} code \cite{singular} to compute the generators of $D^2(f)/S_2$ and $m(D^2(f)/S_2,0).$
\begin{flushleft}

        \texttt{ring T=0,(X,Y,Z),ds;}
        
        \texttt{ring S=0,(x,y,z),ds;}
        
        \texttt{poly p=y4;}
        
        \texttt{poly q=y7+x2y;}
        
        \texttt{map h=T,x,y+z,(y-z)\^\,2;}
        
        \texttt{ideal D2=(p-subst(p,y,z))/(y-z),(q-subst(q,y,z))/(y-z);}\ \ \ // \textit{the ideal defining $D^2(f)$}
        
        \texttt{setring T;}

        \texttt{ideal D2S2=preimage(S,h,D2);}\ \ \ // \textit{the ideal defining $D^2(f)/S_2$}
        
        \texttt{minbase(D2S2);}

        \texttt{mult(std(D2S2));}

        \texttt{//->\,D2S2[1]=YZ+Y3}

        \texttt{//->\,D2S2[2]=64X2+Z3+21Y2Z2+35Y4Z+7Y6}

        \texttt{//->\,4.}
\end{flushleft} 

In this code, we use the coordinates $(X,Y,Z)$ instead of $(x,\sigma_1,\sigma_2).$ In particular, this computation yields $m(D^2(f)/S_2,0)=4$.

For $t\ne0$ sufficiently small, consider the isomorphism given by $x\mapsto -(\sigma_1^3+\sigma_1\sigma_2)/2t.$ Under this isomorphism, $D^2(f_t)/S_2$ is embedded in $\mathbb{C}^2$ as the plane curve $V(t^2(7\sigma_1^6+35\sigma_1^4\sigma_2+21\sigma_1^2\sigma_2^2+\sigma_2^3)+16(\sigma_1^3+\sigma_1\sigma_2)^2)\subset\mathbb{C}^2.$ Therefore, $m(D^2(f_t)/S_2,0)=3.$ Hence, since $m(D^2(f_t)/S_2,0)$ is not constant, we conclude that $D^2(F)/S_2$ is not Whitney equisingular, and a fortiori, is not bi-Lipschitz equisingular.\end{proof}\\

Table \ref{tabela1} summarizes the equisingularity properties of $D(F)$, $F(D(F))$, $D^2(F),$ and $D^2(F)/S_2$ under the hypothesis that $F$ is topologically trivial. Observe that the symbol (\textcolor{blue}{\ding{51}}) in Table \ref{tabela1} indicates that the corresponding equisingularity property holds for the family of curves and the type of equisingularity specified, whereas the symbol (\textcolor{red}{\ding{55}}) indicates that the corresponding property does not hold in general.

\begin{table}[!h]
\centering
{\def\arraystretch{2.6}\tabcolsep=12pt 

\begin{tabular}{ c | c | c | c | c }

\hline
\rowcolor{lightgray}
\hline
\multicolumn{5}{c}{\textbf{Equisingular condition: $F$ is topologically trivial}}\\
\hline
\textbf{Equising. Type}  &  $D(F)$ & $D^2(F)$ &  $F(D(F))$  & $D^2(F)/S_2$ \\

\hline
Topological Trivial     &  \makecell{\textcolor{blue}{\ding{51}} \\ (Lemma \ref{lemma1})} &  \makecell{\textcolor{blue}{\ding{51}} \\ (Lemma \ref{lemma1})}   &  \makecell{\textcolor{blue}{\ding{51}} \\ (Lemma \ref{lemma1})} & \makecell{\textcolor{blue}{\ding{51}} \\ (Lemma \ref{lemma1})} \\

\hline
Whitney equisingular    & \makecell{\textcolor{blue}{\ding{51}} \\ (Lemma \ref{lemma1})} &  \makecell{\textcolor{blue}{\ding{51}} \\ (Lemma \ref{lemma1})}  &  \makecell{\textcolor{red}{\ding{55}} \\ (Lemma \ref{lemma4})}  & \makecell{\textcolor{red}{\ding{55}}  \\ (Prop. \ref{lemma2})} \\

\hline
bi-Lipschitz equisingular   & \makecell{\textcolor{blue}{\ding{51}} \\ (Lemma \ref{lemma1})} &  \makecell{\textcolor{red}{\ding{55}} \\ (Prop. \ref{lemma2}) \\ \textcolor{blue}{\ding{51}} if $f$ is \\ homogeneous \\ (Th. \ref{lemma3})} &   \makecell{\textcolor{red}{\ding{55}} \\ (Lemma \ref{lemma4})} & \makecell{\textcolor{red}{\ding{55}} \\ (Prop. \ref{lemma2})}  \\

\hline

\end{tabular}
}
\caption{Equisingularity of families of double point curves under the topological triviality condition of $F$.}\label{tabela1}
\end{table}

\begin{remark}
    Observe that Theorem \rm{\ref{lemma3}} \textit{and Lemma} \rm{\ref{lemma4}} \textit{in Table} \rm{\ref{tabela1}} \textit{have not yet been presented. However, we believe that this is an appropriate point to present the table, as it provides a general overview of the consequences of assuming that $F$ is only topologically trivial.}
\end{remark}

Having completed the study under the hypothesis that $F$ is topologically trivial, the next step is to assume that $F$ is Whitney equisingular and investigate whether any of the conclusions change.

\begin{lemma}\label{lemma8} Let $f:(\mathbb{C}^2,0)\rightarrow (\mathbb{C}^3,0)$ be a finitely determined map germ, and let $F=(f_t,t)$ be an unfolding of $f$. If $F$ is Whitney equisingular, then $D(F)$, $D^2(F),$ and $F(D(F))$ are Whitney equisingular.
\end{lemma}

\begin{proof}
    Since $F$ is Whitney equisingular, it follows from Lemma \ref{lemma1} that $D(F)$ and $D^2(F)$ are Whitney equisingular. Now, by \cite[Th. 5.3]{marar-juanjo-guille}, the Whitney equisingularity of $F$ is equivalent to the constancy of $\mu(D(f_t),0)$ and $\mu(\gamma_t,0),$ where $\gamma_t$ is the transverse slice of $f_t$ (see \cite{slice} for the definition of the transverse slice). On the other hand, by \cite[Lemma 5.2]{marar-juanjo-guille}, we have \begin{equation}\label{mul}
        \mu(\gamma_t,0)=\mu(\tilde{\mathcal{Y}_t},0)+2m(f_t(D(f_t)),0),
    \end{equation} where $\tilde{\mathcal{Y}}_t=f_t^{-1}(H),$ with $H\subset\mathbb{C}^3$ a generic plane for $f_t$ for all sufficiently small $t\ne0$. Since $\mu(\gamma_t,0)$ is constant, and the invariants in \eqref{mul} are upper semicontinuous, it follows that $m(f_t(D(f_t)),0)$ is constant. Hence, by \cite[Th. III.3]{Briancon} we conclude that $F(D(F))$ is Whitney equisingular.
\end{proof}\\

Although $F(D(F))$ is Whitney equisingular under the assumption that $F$ itself is Whitney equisingular, Proposition \ref{lemma2} shows that $D^2(F)$ need not be bi-Lipschitz equisingular, while $D^2(F)/S_2$ need not be Whitney equisingular.

Note that the map germ $f_t$ in \eqref{eq4} is quasihomogeneous (with distinct weights) for all $t$, and $F=(f_t,t)$ is Whitney equisingular; nevertheless, this alone is not sufficient to ensure that $D^2(F)$ is bi-Lipschitz equisingular. In the next theorem, we provide sufficient conditions for $D^2(F)$ to be bi-Lipschitz equisingular.

Before presenting the next result, let us recall the notion of the Whitney tangent cone $C_5(\mathcal{W},0)$ of an analytic set $\mathcal{W}\subset\mathbb{C}^n$ at the origin, as introduced by Whitney in \cite[Sec. 3]{Whi652} (see also \cite[Sec. 9.1]{Chirka}). Furthermore, let us recall a notion of a generic linear projection from $\mathbb{C}^n$ onto $\mathbb{C}^2$ restricted to a curve $\textbf{X}$ in $\mathbb{C}^n.$ This notion was introduced by Briançon, Galligo, and Granger in \cite[Ch. IV]{Briancon} (see also \cite[Def. 3.1]{Giles}).

\begin{definition}
    (a) Let $\mathcal{W}$ be a representative of the germ of analytic space $(\mathcal{W},0) \subset (\mathbb{C}^n,0)$. We say that a vector $v \in\mathbb{C}^n$ belongs to $C_5(\mathcal{W},0)$ if there exist sequences of points $x_n,y_n \in \mathcal{W}$ and complex numbers $\lambda_n \in \mathbb{C}$ such that $x_n\rightarrow 0$, $y_n\rightarrow 0$ and $\lambda_n\overline{(x_n-y_n)}\rightarrow v$ as $n\rightarrow \infty$ (see \rm\cite[\textit{Sec. $3$}]{Whi652}\textit{).}\\

\noindent \textit{(b) Let $(\textbf{X},0)$ be a germ of curve in $(\mathbb{C}^n,0)$, and let $\pi:(\mathbb{C}^n,0)\rightarrow (\mathbb{C}^2,0)$ be a linear projection. We say that the restriction $\pi|_{(\textbf{X},0)}: (\textbf{X},0)\rightarrow (\mathbb{C}^2,0)$ is a $C_5$-generic projection for the curve germ $(\textbf{X},0)$ if the kernel of $\pi$ intersects $C_5(\textbf{X},0)$ transversally, that is, $\ker(\pi)\cap C_5(\textbf{X},0)=(0, \cdots, 0)$ (see} \rm\cite[\textit{Ch. IV}]{Briancon}\textit{).}
\end{definition}

\begin{remark}\label{planos}
    Let $(\textbf{X},0)$ be a reduced curve germ in $\mathbb{C}^n$ with at least two branches, whose branches are distinct lines. For each line $l_i$ in $(\textbf{X},0)$, denote by $v_i$ the direction vector of $l_i$. For each pair $(i,j),$ let $H_{i,j}$ be the plane generated by $v_i$ and $v_j$ in $\mathbb{C}^n.$ By \rm{\cite[\textit{Th.} 3.2]{c5}}\textit{, we have $$C_5(\textbf{X},0)=\bigcup_{i\ne j}H_{ij},$$ that is, $C_5(\textbf{X},0)$ is the (finite) union of the planes $H_{i,j}.$ In the case where $(\textbf{X},0)$ is a single line, the $C_5$-cone coincides with $(\textbf{X},0).$}
 \end{remark}

\begin{theorem}\label{lemma3} Let $f:(\mathbb{C}^2,0)\rightarrow (\mathbb{C}^3,0)$ be a finitely determined homogeneous map germ, and let $F=(f_t,t)$ be an unfolding of $f$. If $F$ is topologically trivial, then $D^2(F)$ is bi-Lipschitz equisingular.
\end{theorem}

\begin{proof} Since $f$ is homogeneous and finitely determined, $D^2(f)$  (respectively $D(f)$) is a finite union of $r$ distinct lines in $\mathbb{C}^2\times\mathbb{C}^2$ (respectively $\mathbb{C}^2$) passing through the origin.

Let $p:\mathbb{C}^2\times\mathbb{C}^2\longrightarrow\mathbb{C}^2$, $p(x,y,x',y')=(x,y)$, be the projection onto the first factor. Its restriction to \(D^2(f)\) has image \(D(f)\), that is, $p(D^2(f))=D(f)$. Denote by $D(f)^i=p(D^2(f)^i).$ For each $i,$ the component $D(f)^i$ of $D(f)$ can be parametrized by $\varphi_i(u)=(a_iu,\, b_iu),$ for some $a_i,b_i\in\mathbb{C}$ not both zero.

On the other hand, if a pair $(D(f)^i,D(f)^j)$ is such that $f(D(f)^i)=f(D(f)^j)$, then the component $D^2(f)^i$ can be parametrized by \begin{equation}\label{repara}
    \phi_i(u)=(a_iu,\,b_iu,\,c_iu,\,d_iu),
\end{equation} where $\psi_i(u)=(c_iu,\,d_iu)$ is a convenient reparametrization of the line $D(f)^j.$ By an analogous argument, if the restriction of $f$ to $D(f)^i$ is 2-to-1, then $D^2(f)^i$ admits a parametrization of the form \eqref{repara}. However, $\psi_i$ is a reparametrization of $D(f)^i$ itself. Since \(D^2(F)\) is topologically trivial, the number of branches of the fibres is constant. Hence, \(D^2(f_t)\) has exactly \(r\) branches for every sufficiently small \(t\). We will divide the proof into cases.\\

\noindent\textbf{Case $r=1$}. We first consider the case in which \(D^2(f)\) has only one branch. Since \(D^2(F)\) is Whitney equisingular by Lemma 3.1, its multiplicity is constant. As \(D^2(f)\) is a line, we have $m(D^2(f),0)=1$, and hence $m(D^2(f_t),0)=1$ for every sufficiently small \(t\). Thus every \(D^2(f_t)\) is smooth. Since $D^2(F)\to T$ is flat and all its fibres are smooth, it is a smooth morphism. Hence $D^2(F)$ is smooth along the parameter axis and, by the holomorphic constant rank theorem, the family is locally analytically trivial, and in particular it is bi-Lipschitz equisingular.\\

\noindent\textbf{Case $r\geq 2$}. Now, let $v_i=(a_i,\,b_i,\,c_i,\,d_i)$ be the direction vector of $D^2(f)^i$ and, for each $i\ne j,$ consider the plane $H_{i,j}$ generated by $v_i$ and $v_j$. By Remark \ref{planos}, $C_5(D^2(f),0)$ is the finite union of the planes $H_{i,j}.$

Note that the linear projection $p:\mathbb{C}^2\times\mathbb{C}^2\rightarrow\mathbb{C}^2$ is $C_5$-generic for $D^2(f).$ In fact, observe that $(x,y,x',y')\in \ker(p)$ if and only if $x=y=0.$ On the other hand, since $D^2(f)^i$ and $D^2(f)^j$ are distinct lines, we have \begin{equation}\label{det}
    det\left[\begin{array}{cc}
    a_i & b_i \\
    a_j & b_j
\end{array}\right]\ne0.
\end{equation}

Otherwise, $\varphi_i$ is a reparametrization of $\varphi_j,$ which would imply that $D(f)^i=D(f)^j.$ In particular, the dimension of $D^2(p)$ is equal to 1. However, by \cite[Lemma 3.9]{Kleiman}, we know that $D^2(p)$ is homeomorphic to the triple point space in the source of $f.$ Since $f$ is finitely determined, the triple point space in the source of $f$ has dimension zero, which implies that $D^2(p)$ has dimension zero, a contradiction. Hence, the determinant in \eqref{det} is nonzero.

Now, by \eqref{det}, we see that the intersection of $\ker(p)$ with each $H_{i,j}$ consists only of the origin. Thus, the intersection of $\ker(p)$ with $C_5(D^2(f),0)$ consists only of the origin and hence \(p\) is \(C_5\)-generic for \(D^2(f)\). We claim that, after shrinking the parameter space if necessary, the same projection \(p\) is \(C_5\)-generic for \(D^2(f_t)\) for every \(t\in T\).

Indeed, since \(F\) is topologically trivial, \(D^2(F)\) is also topologically trivial. Moreover, as in the proof of Lemma \ref{lemma1}, after shrinking \(T\) the branches of \(D^2(f_t)\) can be labelled consistently with \(t\) and admit parametrizations in family. Thus, for each \(i=1,\ldots,r\), we may choose a holomorphic parametrization
\[
\Phi_i(u,t)
=
\bigl(x_i(u,t),y_i(u,t),x_i'(u,t),y_i'(u,t)\bigr)
\]
of the \(i\)-th branch \(D^2(f_t)^i\), such that $\Phi_i(u,0)=(a_i u,b_i u,c_i u,d_i u)$. Set
\[
v_i(t)
=
\frac{\partial\Phi_i}{\partial u}(0,t)
=
\bigl(a_i(t),b_i(t),c_i(t),d_i(t)\bigr).
\]
Then $v_i(0)=v_i=(a_i,b_i,c_i,d_i)$. In particular, \(v_i(0)\neq0\). Since \(v_i(t)\) depends holomorphically on \(t\), after shrinking \(T\) we have $v_i(t)\neq0$ for every \(i\) and every \(t\in T\). Hence every branch \(D^2(f_t)^i\) is smooth. For \(i\neq j\), consider

\[
\Delta_{ij}(t)
=
\det
\begin{pmatrix}
a_i(t)&b_i(t)\\
a_j(t)&b_j(t)
\end{pmatrix}.
\]

By \eqref{det}, we have that $\Delta_{ij}(0)\neq0$. Since \(\Delta_{ij}(t)\) is holomorphic in \(t\), after shrinking \(T\) once more we obtain that $\Delta_{ij}(t)\neq0$ for every \(i\neq j\) and every \(t\in T\). Therefore, the projected tangent vectors $p(v_i(t))=(a_i(t),b_i(t))$ and $p(v_j(t))=(a_j(t),b_j(t))$ are linearly independent. In particular, the tangent lines of \(D^2(f_t)^i\) and \(D^2(f_t)^j\) are distinct. For each \(i\neq j\), let $H_{i,j}(t)$ be the plane generated by the vectors $v_i(t)$ and $v_j(t)$. Since the branches of $D^2(f_t)$ are smooth and have pairwise distinct tangent lines, the description of the $C_5$-cone for smooth non-tangent branches gives
\[
C_5(D^2(f_t),0)
=
\bigcup_{1\leq i<j\leq r} H_{i,j}(t).
\]

Since $\Delta_{ij}(t) \neq 0$, we obtain that the intersection of $\ker(p)$ with each $H_{i,j}(t)$ consists only of the origin. Thus, the intersection of $\ker(p)$ with $C_5(D^2(f_t),0)$ consists only of the origin and hence \(p\) is \(C_5\)-generic for \(D^2(f_t)\). Finally, since $p(D^2(f_t))=D(f_t)$ and $D(F)$ is topologically trivial, \cite[Cor.~3.6]{Giles} implies that $D^2(F)$ is bi-Lipschitz equisingular.\end{proof}\\

 A natural question that arises at this point is whether the same result holds for $F(D(F))$ in the homogeneous case. Unfortunately, the answer is negative. In \cite{ruas-otoniel}, Ruas and the first author showed that the family $F=(f_t,t),$ where $$f_t(x,y)=(x,y^4,x^5y-5x^3y^3+4xy^5+y^6+ty^7)$$ is topologically trivial. However, $F(D(F))$ is not bi-Lipschitz equisingular. More precisely, they proved the following lemma:
 
\begin{lemma}\label{lemma4}(\rm{\cite[\textit{Ex.} 5.4]{ruas-otoniel}}\textit{) Consider the family of map germs from $(\mathbb{C}^2,0)$ to $(\mathbb{C}^3,0)$, $F=(f_t,t)$ defined by}

\begin{center}
$f_t(x,y)=(x,y^4,x^5y-5x^3y^3+4xy^5+y^6+ty^7)$.
\end{center}

\textit{The map germ $f=f_0$ is finitely determined and $F$ is topologically trivial. However, $F(D(F))$ is neither bi-Lipschitz equisingular nor Whitney equisingular. In particular, $F$ is not Whitney equisingular.}
\end{lemma}

Although the family $f_t$ in Lemma \ref{lemma4} is not Whitney equisingular, our next proposition shows that even when $F=(f_t,t)$ is Whitney equisingular, with $f_0$ a finitely determined quasihomogeneous map germ, $F(D(F))$ may fail to be bi-Lipschitz equisingular.

\begin{proposition}\label{lemma7} Consider the family of map germs $F=(f_t,t)$ from $(\mathbb{C}^2,0)$ to $(\mathbb{C}^3,0)$, where $f_t$ is given by $$f_t(x,y)=(x,y^8+tx^2y^5,y(x^2-y^3)(x^2+y^3)).$$

The map germ $f=f_0$ is finitely determined, and $F$ is Whitney equisingular. However, $F(D(F))$ is not bi-Lipschitz equisingular. In particular, $F$ is not bi-Lipschitz equisingular.

\end{proposition}

\begin{proof}
Using the same \texttt{Singular} code \cite{singular} as in Proposition \ref{lemma2}, together with the Marar--Mond criterion \cite[Th. 2.14]{marar-mond}, one can verify that $f_0$ is finitely determined.
    
    Observe that $f_0$ is quasihomogeneous of type $(3,2;3,16,14).$ Furthermore, $F$ adds only a term of the same weighted degree to the second coordinate function of $f_0.$ Therefore, it follows from \cite[Cor. 1]{Damon} and \cite[Cor. 4.4]{Slice} that $F$ is Whitney equisingular.

    It can be verified that $D(f_0)^1=V(x^2-y^3)$ and $D(f_0)^2=V(x^2+y^3)$ are identification components of $D(f_0)$, that is, the restrictions of $f_0$ to $D(f_0)^1$ and $D(f_0)^2$, respectively, are generically 1-to-1, and $f_0(D(f_0)^1)=f_0(D(f_0)^2).$ Consider the parametrizations $\varphi^1(u)=(u^3,u^2)$ and $\varphi^2(u)=(u^3,-u^2)$ of $D(f_0)^1$ and $D(f_0)^2,$ respectively. Thus, $$f_0(\varphi^1(u))=(u^3,u^{16},0)=f_0(\varphi^2(u)).$$
    
    For $t\neq 0$ sufficiently small, let $\varphi_t^1$ and $\varphi_t^2$ denote parametrizations of the branches corresponding to $\varphi^1$ and $\varphi^2$, respectively. We next show that these parametrizations can be chosen to depend holomorphically on the parameter $t$. We require that $f_t(\varphi_t^1(u)) = f_t(\varphi_t^2(u))$, where $\varphi_t^1(u) = (u^3, \alpha_1(t)u^2)$ and $\varphi_t^2(u) = (u^3, \alpha_2(t)u^2)$ with $(\alpha_1(0),\alpha_2(0)) = (1,-1)$, for some $\alpha_1,\alpha_2\in\mathcal{O}_1.$ Matching the corresponding coefficients produces the system of equations \[
\alpha_1^8+t\alpha_1^5=\alpha_2^8+t\alpha_2^5,\ \ \ \alpha_1-\alpha_1^7=\alpha_2-\alpha_2^7.
\] Now, consider the map $L(\alpha_1,\alpha_2,t)=(\alpha_1^8-\alpha_2^8+t(\alpha_1^5-\alpha_2^5),\,\alpha_1-\alpha_1^7-\alpha_2+\alpha_2^7)$. At $(\alpha_1,\alpha_2,t)=(1,-1,0)$, we have $L(1,-1,0)=(0,0)$, and the Jacobian matrix with respect to $(\alpha_1,\alpha_2)$ is
\[
J_{(\alpha_1,\alpha_2)}L(1,-1,0)=\left[\begin{array}{cc} 8 & 8 \\ -6 & 6 \end{array}\right],\]
which has determinant $96 \neq 0$. By the implicit function theorem, there exist unique holomorphic functions $\alpha_1(t)$ and $\alpha_2(t)$ for small $t$ such that $\alpha_1(0) = 1$ and $\alpha_2(0) = -1$, as required. Since \(\alpha_1(0)=1\) and \(\alpha_2(0)=-1\), after shrinking the parameter space we have $(\alpha_1(t))^3\neq(\alpha_2(t))^3$ for every sufficiently small \(t\). Hence \(\varphi_t^1\) and \(\varphi_t^2\) define two distinct branches of \(D(f_t)\), which are paired by \(f_t\) and holomorphically continue \(D(f_0)^1\) and \(D(f_0)^2\), respectively. In particular, since $D(f_t)^1$ and $D(f_t)^2$ are identification components of $D(f_t),$ we have $$f_t(\varphi_t^j(u))=(u^3,h_{1}(t)u^{16},h_{2}(t)u^{14})$$ where, for either $j=1,2$, $h_1(t)=\alpha_j(t)^8+t\alpha_j(t)^5$ and $h_2(t)=\alpha_j(t)-\alpha_j(t)^7$, these expressions being independent of $j$ by the defining equations above.

Furthermore, applying implicit differentiation with respect to $t$ at $t=0$ in the identity $L(\alpha_1(t),\alpha_2(t),t)=(0,0)$, we obtain $$J_{(\alpha_1,\alpha_2)}L(1,-1,0)\cdot\left[\begin{array}{c}
    \alpha_1'(0)\\
    \alpha_2'(0)
\end{array}\right]+\dfrac{\partial L}{\partial t}(1,-1,0)=\left[\begin{array}{c}
    0\\
    0
\end{array}\right].$$ Since $$\dfrac{\partial L_1}{\partial t}(1,-1,0)=2\ \ \ \text{and}\ \ \  \dfrac{\partial L_2}{\partial t}(1,-1,0)=0,$$ we obtain $\alpha_1'(0)=\alpha_2'(0)=-1/8.$ Consequently, for $h_2(t)=\alpha_1(t) - \alpha_1(t)^7$, we obtain
\[
h_2'(0)=\alpha_1'(0)(1 - 7\alpha_1(0)^6) = \left(-\frac{1}{8}\right)(-6) = \frac{3}{4}.
\]
Since $h_2(0) = 0$ and $h_2'(0) = 3/4 \neq 0$, we have $$h_2(t)=\dfrac{3}{4}t+O(t^2),$$ which implies that $h_2(t) \neq 0$ for all sufficiently small $t \neq 0$.
    
    Set $Y_t:=f_t(D(f_t))$. Since $F(D(F))$ is topologically trivial, the number of irreducible branches of $Y_t$ is constant for $t$ sufficiently small. After shrinking the parameter space if necessary, we label these distinct image branches consistently with $t$ as $Y_t^1,\ldots,Y_t^r$, where $Y_t^1=f_t(D(f_t)^1)=f_t(D(f_t)^2)$. In particular, the common image of the two paired source branches is counted only once. By \cite[Th.~3.5]{Giles}, after shrinking the parameter space if necessary, we may choose a linear projection $p:(\mathbb{C}^3,0)\longrightarrow(\mathbb{C}^2,0)$ which is generic for $Y_t$ for every sufficiently small $t$. Note that for $t=0$, the irreducible plane curve $p(Y_0^1)$ has characteristic exponents $3$ and $16$. On the other hand, for $t\neq0$ sufficiently small, $p(Y_t^1)$ has characteristic exponents $3$ and $14$. By \cite[Cor.~1.2.3]{buch}, we have
\[
\mu(p(Y_t),0)
=
1
+
\sum_{1\leq \ell\leq r}
\bigl(\mu(p(Y_t^\ell),0)-1\bigr)
+
2
\sum_{1\leq \ell_1<\ell_2\leq r}
i\bigl(p(Y_t^{\ell_1}),p(Y_t^{\ell_2});0\bigr),
\]
where $i\bigl(p(Y_t^{\ell_1}),p(Y_t^{\ell_2});0\bigr)$ denotes the intersection multiplicity of the pair of distinct branches $p(Y_t^{\ell_1})$ and $p(Y_t^{\ell_2})$ of \(p(Y_t)\). For the distinguished branch \(Y_t^1\), we have
\[
\mu(p(Y_0^1),0)=(3-1)(16-1)=30,
\]
whereas, for every sufficiently small \(t\neq0\),
\[
\mu(p(Y_t^1),0)=(3-1)(14-1)=26.
\]
Hence $\mu(p(Y_t^1),0)<\mu(p(Y_0^1),0)$. By the upper semicontinuity of the Milnor numbers of the branches and of the pairwise intersection multiplicities, for \(t\) sufficiently small we have
\[
\mu(p(Y_t^\ell),0)
\leq
\mu(p(Y_0^\ell),0)
\]
for every \(1\leq \ell\leq r\), and
\[
i\bigl(p(Y_t^{\ell_1}),p(Y_t^{\ell_2});0\bigr)
\leq
i\bigl(p(Y_0^{\ell_1}),p(Y_0^{\ell_2});0\bigr)
\]
for every \(1\leq\ell_1<\ell_2\leq r\).

Since the inequality is strict for the branch \(Y_t^1\), it follows that
\[
\mu(p(Y_t),0)
<
\mu(p(Y_0),0)
\]
for every sufficiently small \(t\neq0\). Therefore, the projected family $(p\times \operatorname{id})(F(D(F)))$ is not topologically trivial. Hence, by \cite[Cor.~3.6]{Giles}, $F(D(F))$ is not bi-Lipschitz equisingular. Finally, since the bi-Lipschitz equisingularity of $F$ would imply the bi-Lipschitz equisingularity of $F(D(F))$, we conclude that $F$ is not bi-Lipschitz equisingular.\end{proof}

\begin{remark}
    The reliance on Puiseux parametrizations and characteristic exponents in Proposition \rm{\ref{lemma7}} \textit{underscores the fundamental role that branch-carried analytic data play in determining local invariants. From a modern geometric perspective, the interplay between Puiseux data, analytic semiroots, and differential forms has been further developed by Cano, Corral, and Senovilla-Sanz} \rm{\cite{Cano}}\textit{, who demonstrate how enlarged standard bases and dicritical foliations preserve finer analytic invariants across equisingular packages of plane branches.}
\end{remark}

Table \ref{tabela2} summarizes the equisingularity properties of $D(F)$, $F(D(F))$, $D^2(F),$ and $D^2(F)/S_2$ under the assumption that $F$ is Whitney equisingular.

\begin{table}[!h]
\centering
{\def\arraystretch{1.6}\tabcolsep=12pt 

\begin{tabular}{ c | c | c | c | c }

\hline
\rowcolor{lightgray}
\multicolumn{5}{c}{\textbf{Equisingular condition: $F$ is Whitney equisingular}}\\
\hline
\textbf{Equising. Type}  &  $D(F)$ & $D^2(F)$ &  $F(D(F))$  & $D^2(F)/S_2$ \\
\hline
Topological Trivial     &  \makecell{\textcolor{blue}{\ding{51}} \\ (Lemma \ref{lemma1})} &  \makecell{\textcolor{blue}{\ding{51}} \\ (Lemma \ref{lemma1})}   &  \makecell{\textcolor{blue}{\ding{51}} \\ (Lemma \ref{lemma1})} & \makecell{\textcolor{blue}{\ding{51}} \\ (Lemma \ref{lemma1})} \\
\hline
Whitney equisingular    & \makecell{\textcolor{blue}{\ding{51}} \\ (Lemma \ref{lemma1})} &  \makecell{\textcolor{blue}{\ding{51}} \\ (Lemma \ref{lemma1})}  &  \makecell{\textcolor{blue}{\ding{51}} \\ (Lemma \ref{lemma8})}  & \makecell{\textcolor{red}{\ding{55}} \\ (Prop. \ref{lemma2})} \\

\hline
bi-Lipschitz equisingular   & \makecell{\textcolor{blue}{\ding{51}} \\ (Lemma \ref{lemma1})} &  \makecell{\textcolor{red}{\ding{55}} \\ (Prop. \ref{lemma2}) \\ \textcolor{blue}{\ding{51}} if $f$ is \\ homogeneous \\ (Th. \ref{lemma3})} &   \makecell{\textcolor{red}{\ding{55}} \\ (Prop. \ref{lemma7})} & \makecell{\textcolor{red}{\ding{55}} \\ (Prop. \ref{lemma2})}  \\

\hline

\end{tabular}
}
\caption{Equisingularity of families of double point curves under the Whitney equisingularity of $F$.}\label{tabela2}
\end{table}

Let us now assume that $F=(f_t,t)$ is a bi-Lipschitz unfolding of a finitely determined map germ from $(\mathbb{C}^2,0)$ to $(\mathbb{C}^3,0)$. The following lemma is a consequence of a result of Brasselet, Ruas, and Thuy (see \cite[Prop. 3.1]{thuy}).

\begin{lemma}\label{lemma9} Let $f:(\mathbb{C}^2,0)\rightarrow (\mathbb{C}^3,0)$ be a finitely determined map germ, and let $F=(f_t,t)$ be an unfolding of $f$. If $F$ is bi-Lipschitz equisingular, then $D(F)$, $D^2(F),$ and $F(D(F))$ are bi-Lipschitz equisingular.
\end{lemma}

\begin{proof}
    The result follows from Lemma \ref{lemma1} and \cite[Prop. 3.1]{thuy}.
\end{proof}\\

To conclude this section, we propose a conjectural criterion for determining whether an unfolding $F$ of a finitely determined map germ $f$ from $(\mathbb{C}^2,0)$ to $(\mathbb{C}^3,0)$ is bi-Lipschitz equisingular.

\begin{conjecture}\label{conjectura}
Let $f:(\mathbb{C}^2,0)\rightarrow(\mathbb{C}^3,0)$ be a finitely determined map germ, and let $F$ be a $1$-parameter unfolding of $f.$ Denote by $\mathcal{P}_1(f_t(\mathbb{C}^2))$ the polar curve of $f_t$ (see \cite{gaffney} for the definition of the polar curve). If the family of curves 
\begin{center}
 $\mathcal{C}(f_t):=\mathcal{P}_1(f_t(\mathbb{C}^2))\cup f_t(D(f_t))$
\end{center}
\noindent is bi-Lipschitz equisingular, then $F$ is bi-Lipschitz equisingular.
\end{conjecture}

Table \ref{tabela3} summarizes our results on the equisingularity of $D(F)$, $F(D(F))$, $D^2(F),$ and $D^2(F)/S_2$ under the hypothesis that $F$ is bi-Lipschitz equisingular. If Conjecture \ref{conjectura} is true, it will provide a criterion for determining whether \(F\) is bi-Lipschitz equisingular, and hence a systematic way to search for counterexamples to the remaining open cases in Table \ref{tabela3}.

\begin{table}[!h]
\centering
{\def\arraystretch{1.6}\tabcolsep=12pt 

\begin{tabular}{ c | c | c | c | c }

\hline
\rowcolor{lightgray}
\multicolumn{5}{c}{\textbf{Equisingular condition: $F$ is bi-Lipschitz equisingular}}\\
\hline
\textbf{Equising. Type}  &  $D(F)$ & $D^2(F)$ &  $F(D(F))$  & $D^2(F)/S_2$ \\

\hline
Topological Trivial     &  \makecell{\textcolor{blue}{\ding{51}} \\ (Lemma \ref{lemma1})} &  \makecell{\textcolor{blue}{\ding{51}} \\ (Lemma \ref{lemma1})}   &  \makecell{\textcolor{blue}{\ding{51}} \\ (Lemma \ref{lemma1})} & \makecell{\textcolor{blue}{\ding{51}} \\ (Lemma \ref{lemma1})} \\

\hline
Whitney equisingular    & \makecell{\textcolor{blue}{\ding{51}} \\ (Lemma \ref{lemma1})} &  \makecell{\textcolor{blue}{\ding{51}} \\ (Lemma \ref{lemma1})}  &  \makecell{\textcolor{blue}{\ding{51}} \\ (Lemma \ref{lemma8})}  & \textbf{?} \\

\hline
bi-Lipschitz equisingular   & \makecell{\textcolor{blue}{\ding{51}} \\ (Lemma \ref{lemma1})} &  \makecell{\textcolor{blue}{\ding{51}} \\ (Lemma \ref{lemma9})} &   \makecell{\textcolor{blue}{\ding{51}} \\ (Lemma \ref{lemma9})} & \textbf{?}  \\

\hline

\end{tabular}
}
\caption{Equisingularity of families of double point curves under the bi-Lipschitz equisingularity of $F$.}\label{tabela3}
\end{table}

\section{Families of curves of Henry--type}\label{henrytype}

$ \ \ \ \ $ An important family of curves was introduced by J. P. G. Henry. In the following example, we recall this family of curves (see, for instance, \cite[Ex. V.1]{Briancon}):

\begin{example}[J. P. G. Henry]\label{henry}
    For a fixed $t\in\mathbb{C},$ define $\textbf{X}_t\subset \mathbb{C}^3$ by the equations $$xy-tz=0\ \ \ \ \text{and}\ \ \ \ z^6+x^{15}+y^{10}=0.$$ 
    
    Each $\textbf{X}_t$  is a quasihomogeneous complete intersection curve singularity. For every sufficiently small $t$, we have $\mu(\textbf{X}_t,0)=126,$ and hence the family $\textbf{X}_t$ is topologically trivial. However, $m(\textbf{X}_0,0)=12,$ while $m(\textbf{X}_t,0)=10$ for all $t\ne0$. Therefore, the family $\textbf{X}_t$ is not Whitney equisingular. 
\end{example}

Some works have used families of curves of this type to provide counterexamples to certain problems (see, for instance, \cite[Ex. V.1]{Briancon}, \cite[Ex. 7.2.1]{buch}, and \cite[Ex. 4.21]{victor-juanjo}). Inspired by Example \ref{henry}, we say that a family of curves $\textbf{X}_t$ is of \textit{Henry--type} if each $\textbf{X}_t$ is a quasihomogeneous complete intersection curve, and the family is topologically trivial but not Whitney equisingular. In Proposition \ref{lemma2}, the family of curves $D^2(f_t)/S_2=V(\sigma_1^3+\sigma_1\sigma_2+2tx,\,7\sigma_1^6+35\sigma_1^4\sigma_2+21\sigma_1^2\sigma_2^2+\sigma_2^3+64x^2)$ is of Henry--type, and provides a counterexample to one of the problems that we consider in this work.

The construction of these $\mu$-constant families is based on controlling how weighted-degree perturbation terms alter topological and geometric invariants. It is worth noting that while our focus in this section is on isolating $\mu$-constant strata to produce Henry--type examples of double point curves, a complementary perspective on the spectrum of attainable Milnor numbers and their jump phenomena under semi-quasihomogeneous deformations is developed by Michalska and Walewska \cite{Michalska}.

Since families of curves of this type are a powerful tool for constructing counterexamples, we conclude this section by presenting infinite new families of curves of Henry--type, as in Proposition \ref{lemma2}.

\begin{proposition}\label{prop1}
    For each odd integer $k\ge1$, let $f^k_t:(\mathbb{C}^2,0)\rightarrow(\mathbb{C}^3,0)$ be the family of quasihomogeneous, corank 1 map germs given by $$f^k_t(x,y)=(x,y^4+txy,y^{6k+1}+x^{2k}y).$$
    
    Then $f^k_0$ is finitely determined, and $D^2(f_0^k)/S_2$ is a quasihomogeneous complete intersection. Furthermore, $$\mu(D^2(f_t^k)/S_2,0)=(3k-1)(6k-1)$$ for all sufficiently small $t$, while $$m(D^2(f_0^k)/S_2,0)=4k$$ and $$m(D^2(f_t^k)/S_2,0)=3k$$ for all sufficiently small $t\ne0$. In particular, $D^2(f^{k}_t)/S_2$ is a family of curves of Henry--type.
\end{proposition}

\begin{proof}
    Fix an odd integer $k\ge1.$ Note that $f_0^k$ is a reflection map (see \cite{[5]}). Then, by \cite[Th. 5.2]{zampiva} (see also \cite[Prop. 6.3]{Milena}) we have that \begin{eqnarray*}
        D(f_0^k)&=&V\left(\dfrac{1}{4y^3}\prod_{i=1}^3(y^{6k+1}+x^{2k}y-(\xi^i y)^{6k+1}-x^{2k}(\xi^i y))\right)\\
        &=&V\left(\prod_{i=1}^3((1-(\xi^i)^{6k+1})y^{6k}+(1-\xi^i)x^{2k})\right),
    \end{eqnarray*}where $\xi\in\mathbb{C}$ is a primitive fourth root of unity. Since $k\ge1$ is an odd integer, we can write $k=2\tilde{k}+1,$ for some integer $\tilde{k}\ge0$. Then $$(\xi^i)^{6k+1}=(\xi^{i})^{12\tilde{k}+7}=\xi^{3i}$$ and hence \begin{eqnarray*}
        D(f_0^k)&=&V\Big([(1-\xi^{3})y^{6k}+(1-\xi)x^{2k}][(1-\xi^2)y^{6k}+(1-\xi^2)x^{2k}][(1-\xi)y^{6k}+(1-\xi^3)x^{2k}]\Big)\\
        &=&V\Big([(1+\xi+\xi^2)y^{6k}+x^{2k}][y^{6k}+x^{2k}][y^{6k}+(1+\xi+\xi^2)x^{2k}]\Big)\\
        &=&V\Big((-\xi^3y^{6k}+x^{2k})(y^{6k}+x^{2k})(y^{6k}-\xi^3x^{2k})\Big),
    \end{eqnarray*}where the last equality holds by the identity $1+\xi+\xi^2+\xi^3=0.$

    A straightforward calculation shows that the Milnor numbers of $V(-\xi^3y^{6k}+x^{2k}), V(y^{6k}+x^{2k}),$ and $V(y^{6k}-\xi^3x^{2k})$ are finite by computing the $\mathbb{C}$-vector space dimension of the Milnor algebra for each curve. To conclude that $\mu(D(f_0^k),0)$ is finite, we need to check that $V(-\xi^3y^{6k}+x^{2k}), V(y^{6k}+x^{2k}),$ and $V(y^{6k}-\xi^3x^{2k})$ have no common component pairwise. Since $f_0^k$ is a quasihomogeneous map germ of type $(3,1;3,4,6k+1),$ each component $D(f_0^k)^j$ of $D(f_0^k)$ admits a parametrization of the form $\varphi_j^k(u)=(u^3,\alpha_{j,k}u),$ for some $\alpha_{j,k}\in\mathbb{C}\setminus\lbrace0\rbrace.$

    Suppose that $V(-\xi^3y^{6k}+x^{2k})$ and $V(y^{6k}+x^{2k})$ have a common component. This common component is equal to $D(f_0^k)^j$ for some $j.$ In particular, using the parametrization $\varphi_j^k(u)=(u^3,\alpha_{j,k}u)$ of this component, we obtain $$(-\xi^3\alpha_{j,k}^{6k}+1)u^{6k}=0=(\alpha_{j,k}^{6k}+1)u^{6k}$$ for all sufficiently small $u,$ which implies that $\xi^3=-1.$ Hence, $\xi^2=1,$ a contradiction. By an analogous argument, one can check that $V(y^{6k}-\xi^3x^{2k})$ and $V(y^{6k}+x^{2k})$ have no common component.

    Now, suppose that $V(-\xi^3y^{6k}+x^{2k})$ and $V(y^{6k}-\xi^3x^{2k})$ have a common component. As in the previous case, let $\varphi_j^k(u)=(u^3,\alpha_{j,k}u)$ be a parametrization of the common component $D(f_0^k)^j$ for some $j.$ Then $$(-\xi^3\alpha_{j,k}^{6k}+1)u^{6k}=0=(\alpha_{j,k}^{6k}-\xi^3)u^{6k}.$$ This implies that $(\xi^3+1)\alpha_{j,k}^{6k}=\xi^3+1.$ Since $\xi^3+1\ne0,$ we conclude that $\alpha_{j,k}^{6k}=1.$ Now, since the equality $(-\xi^3\alpha_{j,k}^{6k}+1)u^{6k}=0$ holds for all sufficiently small $u,$ we obtain $\xi^3=1,$ a contradiction.

    Hence, since $\mu(D(f_0^k),0)$ is finite, it follows from \cite[Cor. 3.5]{marar-mond} that $f_0^k$ is finitely determined. Moreover, note that the unfolding adds only a term of the same weighted degree in the second coordinate function of $f_0^k$. Hence, it follows from \cite[Cor. 1]{Damon} that the family of map germs $f_t^k$ is topologically trivial. In particular, $f_t^k$ is finitely determined for all sufficiently small $t\ne0$, and it follows from \cite[Th. 2.14]{marar-mond} and \cite[Prop. 2.4.5]{teseguille} that $D^2(f_t^k)/S_2$ is a complete intersection of dimension 1 for all sufficiently small $t$.
    
    Now, by Proposition \ref{geradores}, we obtain\begin{eqnarray*}
        D^2(f_t^k)/S_2&=&V\left(\overline{P}_3(\sigma_1,\sigma_2)+tx,\,\overline{P}_{6k}(\sigma_1,\sigma_2)+x^{2k}\right)\\
        &=&V\left(\sigma_1^3+\sigma_1\sigma_2+2tx,\,\dfrac{1}{2^{6k}}\sum_{i=0}^{3k}\binom{6k+1}{2i+1}\sigma_1^{6k-2i}\sigma_2^i+x^{2k}\right)\\
        &=&V\left(\sigma_1^3+\sigma_1\sigma_2+2tx,\,\sum_{i=0}^{3k}\binom{6k+1}{2i+1}\sigma_1^{6k-2i}\sigma_2^i+2^{6k}x^{2k}\right),
    \end{eqnarray*}for all $t$. Moreover, it is not difficult to check that $D^2(f_t^k)/S_2$ is a quasihomogeneous curve of type $(3,1,2;3,6k).$
    
    Since $f_t^k$ is a quasihomogeneous map germ of type $(3,1;3,4,6k+1),$ Mond's formulas \cite{form} yield $$C(f_t^k)=6k,\ \ T(f_t^k)=2k(6k-1),\ \ \text{and}\ \ \mu(D(f_t^k),0)=(6k-1)(18k-1)$$ for all sufficiently small $t$.

    Now, by the Marar--Mond formulas \cite{marar-mond}, we obtain\begin{eqnarray*}
        \mu(D^2(f_t^k)/S_2,0)&=&\dfrac{1}{2}(\mu(D(f_t^k),0)-6T(f_t^k)-C(f_t^k)+1)\\
        &=&\dfrac{1}{2}((6k-1)(18k-1)-12k(6k-1)-6k+1)\\
        &=&\dfrac{1}{2}((6k-1)(6k-1)-(6k-1))\\
        &=&(3k-1)(6k-1)
    \end{eqnarray*} for all $t$ sufficiently small.

    To complete the proof, we compute the multiplicity of $D^2(f_t^k)/S_2.$ For $t\ne0$ sufficiently small, consider the isomorphism given by $x\mapsto-(\sigma_1^3+\sigma_1\sigma_2)/2t.$ It follows that $D^2(f_t^k)/S_2$ is embedded in $\mathbb{C}^2$ as the plane curve $V(t^{2k}\sum_{i=0}^{3k}\binom{6k+1}{2i+1}\sigma_1^{6k-2i}\sigma_2^i+2^{4k}(\sigma_1^3+\sigma_1\sigma_2)^{2k})\subset\mathbb{C}^2.$ Therefore, $m(D^2(f_t^k)/S_2,0)=3k.$
    
    On the other hand, let $H=V(ax+b\sigma_1+c\sigma_2)\subset\mathbb{C}^3$ be a generic plane for $D^2(f_0^k)/S_2$. By the genericity of $H,$ we may assume that $a\ne0.$ In this case, we have $H=V(x-\alpha\sigma_1-\beta\sigma_2),$ where $\alpha=-b/a$ and $\beta=-c/a.$ Then, \begin{eqnarray*}
        m(D^2(f_{0}^k)/S_2,0)&=&\dim_\mathbb{C}\dfrac{\mathcal{O}_3}{\langle\sigma_1^3+\sigma_1\sigma_2,\,\sum_{i=0}^{3k}\binom{6k+1}{2i+1}\sigma_1^{6k-2i}\sigma_2^i+2^{6k}x^{2k},\,x-\alpha\sigma_1-\beta\sigma_2\rangle}\\
        &=&\dim_\mathbb{C}\dfrac{\mathcal{O}_2}{\langle\sigma_1^3+\sigma_1\sigma_2,\,\sum_{i=0}^{3k}\binom{6k+1}{2i+1}\sigma_1^{6k-2i}\sigma_2^i+2^{6k}(\alpha\sigma_1+\beta\sigma_2)^{2k}\rangle}\\
        &=&4k,
    \end{eqnarray*}where the last equality follows by noting that the tangent of $$V\left(\sigma_1^3+\sigma_1\sigma_2,\,\sum_{i=0}^{3k}\binom{6k+1}{2i+1}\sigma_1^{6k-2i}\sigma_2^i+2^{6k}(\alpha\sigma_1+\beta\sigma_2)^{2k})\right)$$ is given by $V(\sigma_1\sigma_2,\,(\alpha\sigma_1+\beta\sigma_2)^{2k})$, and by applying an appropriate version of Bézout's theorem (see, for instance, \cite[Sec. 10.3, Th. 2']{Chirka}). This concludes the proof.
\end{proof}

\begin{remark}
    Note that if $k$ is even in Proposition \rm{\ref{prop1}}\textit{, then we can write $k=2\tilde{k},$ for some integer $\tilde{k}\ge1.$ In this case, $$(\xi^i)^{6k+1}=(\xi^i)^{12\tilde{k}+1}=\xi^i$$ and we obtain}\begin{eqnarray*}
        D(f_0^k)&=&V\Big([(1-\xi)y^{6k}+(1-\xi)x^{2k}][(1-\xi^2)y^{6k}+(1-\xi^2)x^{2k}][(1-\xi^3)y^{6k}+(1-\xi^3)x^{2k}]\Big)\\
        &=&V((y^{6k}+x^{2k})^3),
    \end{eqnarray*}\textit{which implies that $D(f_0^k)$ is not a reduced curve. Hence, by} \rm{\cite[\textit{Cor.} 3.5]{marar-mond}} \textit{we conclude that $f_0^k$ is not finitely determined. This explains why $k$ is assumed to be odd.}
\end{remark}

\begin{proposition}\label{prop2}
    For each integer $k\ge2$, let $g^k_t:(\mathbb{C}^2,0)\rightarrow(\mathbb{C}^3,0)$ be the family of quasihomogeneous, corank 1 map germs given by $$g^k_t(x,y)=(x,y^{2k}+txy,y^{4k-1}+x^{2}y).$$
    
    Then $g^k_0$ is finitely determined, and $D^2(g_0^k)/S_2$ is a quasihomogeneous complete intersection. Furthermore, $$\mu(D^2(g_t^k)/S_2,0)=2(k-1)(4k-3)$$ for all $t$ sufficiently small, while $$m(D^2(g_0^k)/S_2,0)=2k$$ and $$m(D^2(g_t^k)/S_2,0)=2k-1$$ for all sufficiently small $t\ne0$. In particular, $D^2(g^{k}_t)/S_2$ is a family of curves of Henry--type.
\end{proposition}

\begin{proof}
    Fix an integer $k\ge2.$ Note that $g_0^k$ is a reflection map (see \cite{[5]}). Then, by \cite[Th. 5.2]{zampiva} (see also \cite[Prop. 6.3]{Milena}) we have\begin{eqnarray*}
        D(g_0^k)&=&V\left(\dfrac{1}{2ky^{2k-1}}\prod_{i=1}^{2k-1}(y^{4k-1}+x^2y-(\xi^iy)^{4k-1}-x^2(\xi^iy))\right)\\
        &=&V\left(\prod_{i=1}^{2k-1}((1-(\xi^i)^{4k-1})y^{4k-2}+(1-\xi^i)x^2)\right)\\
        &=&V\left(\prod_{i=1}^{2k-1}((1-\xi^{-i})y^{4k-2}+(1-\xi^i)x^2)\right).
    \end{eqnarray*}where $\xi\in\mathbb{C}$ is a primitive $2k$-th root of unity. Note that $1-\xi^{-i}=-\xi^{-i}(1-\xi^i).$ Then $$D(g_0^k)=V\left(\prod_{i=1}^{2k-1}(-\xi^{-i}y^{4k-2}+x^2)\right).$$

    For each $i=1,\dots,2k-1,$ a straightforward calculation shows that the Milnor number of $V(-\xi^{-i}y^{4k-2}+x^2)$ is finite by computing the $\mathbb{C}$-vector space dimension of its Milnor algebra. To conclude that $\mu(D(g_0^k),0)$ is finite, we need to check that the curves $V(-\xi^{-i_1}y^{4k-2}+x^2)$ and $V(-\xi^{-i_2}y^{4k-2}+x^2)$ have no common component for $i_1\ne i_2$, where $i_1,i_2\in\left\{1,\dots,2k-1\right\}.$ Note that $g_0^k$ is quasihomogeneous of type $(2k-1,1;2k-1,2k,4k-1)$. Therefore, if these two curves have a common component, then this component is equal to $D(g_0^k)^j$ for some $j,$ and it has a parametrization of the form $\varphi_j^k(u)=(\alpha_{j,k}u^{2k-1},u)$ for some $\alpha_{j,k}\in\mathbb{C}\setminus\lbrace0\rbrace.$ Using this parametrization, we have $$(-\xi^{-i_1}+\alpha_{j,k}^2)u^{4k-2}=0=(-\xi^{-i_2}+\alpha_{j,k}^2)u^{4k-2}$$ for all sufficiently small $u.$ Thus, we obtain $i_1=i_2,$ a contradiction.

    Hence, $\mu(D(g_0^k),0)$ is finite. By \cite[Cor. 3.5]{marar-mond}, we conclude that $g_0^k$ is finitely determined. At this point, the proof is analogous to that of Proposition \ref{prop1}. First, for $t\ne0$, observe that $g_t^k$ adds only a term of the same weighted degree in the second coordinate function of $g_0^k$. Thus, it follows from \cite[Cor. 1]{Damon} that the family of map germs $g_t^k$ is topologically trivial. In particular, $g_t^k$ is finitely determined for all sufficiently small $t\ne0$, and it follows from \cite[Th. 2.14]{marar-mond} and \cite[Prop. 2.4.5]{teseguille} that $D^2(g_t^k)/S_2$ is a complete intersection of dimension 1 for all sufficiently small $t$.

    Now, by Proposition \ref{geradores}, we obtain\begin{eqnarray*}
        D^2(g_t^k)/S_2&=&V\left(\overline{P}_{2k-1}(\sigma_1,\sigma_2)+tx,\,\overline{P}_{4k-2}(\sigma_1,\sigma_2)+x^2\right)\\
        &=&V\left(\dfrac{1}{2^{2k-1}}\sum_{i=0}^{k-1}\binom{2k}{2i+1}\sigma_1^{2k-1-2i}\sigma_2^i+tx,\,\dfrac{1}{2^{4k-2}}\sum_{i=0}^{2k-1}\binom{4k-1}{2i+1}\sigma_1^{4k-2-2i}\sigma_2^i+x^2\right)\\
        &=&V\left(\sum_{i=0}^{k-1}\binom{2k}{2i+1}\sigma_1^{2k-1-2i}\sigma_2^i+2^{2k-1}tx,\,\sum_{i=0}^{2k-1}\binom{4k-1}{2i+1}\sigma_1^{4k-2-2i}\sigma_2^i+2^{4k-2}x^2\right)
    \end{eqnarray*}for all $t$. Furthermore, it is not difficult to check that $D^2(g_t^k)/S_2$ is a quasihomogeneous curve of type $(2k-1,1,2;2k-1,4k-2).$

    Since $g_t^k$ is a quasihomogeneous map germ of type $(2k-1,1;2k-1,2k,4k-1),$ Mond's formulas \cite{form} give $$C(g_t^k)=2(2k-1),\ \ T(g_t^k)=\dfrac{2}{3}(4k-3)(k-1)(2k-1) ,\ \ \text{and}\ \ \mu(D(g_t^k),0)=(4k-3)(8k^2-8k+1)$$ for all sufficiently small $t$.

    Now, by the Marar--Mond formulas \cite{marar-mond}, we obtain\begin{eqnarray*}
        \mu(D^2(g_t^k)/S_2,0)&=&\dfrac{1}{2}(\mu(D(g_t^k),0)-6T(g_t^k)-C(g_t^k)+1)\\
        &=&\dfrac{1}{2}\left((4k-3)(8k^2-8k+1)-4(4k-3)(k-1)(2k-1)-2(2k-1)+1\right)\\
        &=&\dfrac{1}{2}(4k-3)(8k^2-8k+1-8k^2+12k-4-1)\\
        &=&2(k-1)(4k-3),
    \end{eqnarray*} for all sufficiently small $t$.

    To complete the proof, we compute the multiplicity of $D^2(g_t^k)/S_2.$ For $t\ne0$ sufficiently small, we consider the isomorphism given by $x\mapsto-(\sum_{i=0}^{k-1}\binom{2k}{2i+1}\sigma_1^{2k-1-2i}\sigma_2^i)/2^{2k-1}t,$ and conclude that $D^2(g_t^k)/S_2$ is embedded in $\mathbb{C}^2$ as the plane curve $V(t^{2}\sum_{i=0}^{2k-1}\binom{4k-1}{2i+1}\sigma_1^{4k-2-2i}\sigma_2^i+(\sum_{i=0}^{k-1}\binom{2k}{2i+1}\sigma_1^{2k-1-2i}\sigma_2^i)^{2})\subset\mathbb{C}^2.$ Therefore, $m(D^2(g_t^k)/S_2,0)=2k-1.$
    
    On the other hand, let $H=V(ax+b\sigma_1+c\sigma_2)\subset\mathbb{C}^3$ be a generic plane for $D^2(g_0^k)/S_2$. By the genericity of $H,$ we may assume that $a\ne0.$ In this case, we have $H=V(x-\alpha\sigma_1-\beta\sigma_2),$ where $\alpha=-b/a$ and $\beta=-c/a.$ Then, \begin{eqnarray*}
        m(D^2(g_{0}^k)/S_2,0)&=&\dim_\mathbb{C}\dfrac{\mathcal{O}_3}{\langle\sum_{i=0}^{k-1}\binom{2k}{2i+1}\sigma_1^{2k-1-2i}\sigma_2^i,\,\sum_{i=0}^{2k-1}\binom{4k-1}{2i+1}\sigma_1^{4k-2-2i}\sigma_2^i+2^{4k-2}x^2,\,x-\alpha\sigma_1-\beta\sigma_2\rangle}\\
        &=&\dim_\mathbb{C}\dfrac{\mathcal{O}_2}{\langle\sum_{i=0}^{k-1}\binom{2k}{2i+1}\sigma_1^{2k-1-2i}\sigma_2^i,\,\sum_{i=0}^{2k-1}\binom{4k-1}{2i+1}\sigma_1^{4k-2-2i}\sigma_2^i+2^{4k-2}(\alpha\sigma_1+\beta\sigma_2)^2\rangle}\\
        &=&2k,
    \end{eqnarray*}where the last equality follows by noting that the tangent of $$V\left(\sum_{i=0}^{k-1}\binom{2k}{2i+1}\sigma_1^{2k-1-2i}\sigma_2^i,\,\sum_{i=0}^{2k-1}\binom{4k-1}{2i+1}\sigma_1^{4k-2-2i}\sigma_2^i+2^{4k-2}(\alpha\sigma_1+\beta\sigma_2)^2\right)$$ is given by $V(\sigma_1\sigma_2^{k-1},\,(\alpha\sigma_1+\beta\sigma_2)^{2})$, and by applying an appropriate version of Bézout's theorem (see, for instance, \cite[Sec. 10.3, Th. 2']{Chirka}).
\end{proof}

\section{Some double point formulas}\label{Sec6}

$ \ \ \ \ $ Let $f:(\mathbb{C}^n,0)\rightarrow(\mathbb{C}^{2n-1},0)$ be a finitely determined map germ, with $n\ge3.$ In this section, we present a convenient analytic structure for the curves $D(f)$ and $f(D(f)),$ and we also provide a positive answer to Question 3 in the Introduction, i.e., we show that 

\begin{equation}\label{eq30}
\mu(D(f))=\mu(D^2(f))  \ \ \ \text{and} \ \ \ \mu(D(f))=2\mu(f(D(f)))+C(f)-1,
\end{equation}

\noindent provided that suitable analytic structures are considered for $D(f)$ and $f(D(f))$.

\begin{remark}
    As in the case of map germs from $(\mathbb{C}^2,0)$ to $(\mathbb{C}^3,0)$, in the case of map germs from $(\mathbb{C}^n,0)$ to $(\mathbb{C}^{2n-1},0)$ we can also define the double point curve $D^2(f)$ in $\mathbb{C}^n \times \mathbb{C}^n$ with the analytic structure given by Mond's ideal and the quotient double point space $D^2(f)/S_2$ with the analytic quotient structure. If $f$ is finitely determined, then $D^2(f)$ and $D^2(f)/S_2$ are reduced curves (see \rm{\cite{victor-juanjo}}\textit{).}
\end{remark}

Next, we describe a convenient analytic structure for $D(f)$ and $f(D(f))$ that will enable us to prove the formulas in (\ref{eq30}). Our construction is inspired by Gaffney's ideas regarding the definition of analytic structures for certain spaces of multiple points (see for instance \cite[Def. 9.3]{mond-juanjo-livro}).\\

Since we are in the range of nice dimensions in the sense of Mather, it is well known that a stabilization of a finitely determined map germ always exists. Let $f:(\mathbb{C}^n,0) \rightarrow (\mathbb{C}^{2n-1},0)$ be a finitely determined map germ, with $n \geq 3$. Now, let $F=(f_t,t)$ be a 1-parameter stabilization of $f=f_0$. Consider the complex spaces $D(F)_{\mathrm{red}}$ in $(\mathbb{C}^n\times\mathbb{C},0),$ and $F(D(F))_{\mathrm{red}}$ in $(\mathbb{C}^{2n-1}\times\mathbb{C},0)$, both spaces equipped with the reduced structure. Then we define the analytic space $D_F(f)$ as the fibre over $0 \in T$ of the projection $p: D(F) \rightarrow T$, where $T$ is the parameter space. In a similar way, we define the analytic space $f_F(D(f))$ as the fibre over $0 \in T$ of the projection $p: F(D(F)) \rightarrow T$. In other words, if $D(F)_{\mathrm{red}}=V(I_F)$ and $F(D(F))_{\mathrm{red}}=V(J_F)$ then $D_F(f)=V(I_F,t)$ and $f_F(D(f))=V(J_F,t)$.\\


Clearly, the analytic structures of $D_F(f)$ and $f_F(D(f))$ may
depend on the chosen stabilization $F$. For this reason, we retain the
subscript $F$ when referring to these analytic spaces. However, we will
show below that the numerical invariants $\mu$ and $\epsilon$ associated
with these spaces do not depend on the choice of $F$, and, moreover, they are $\mathcal{A}$-invariants of $f$. Before, we recall the notion of $\epsilon$-invariant and the Milnor number for generically reduced curves, as defined in \cite{Brucker}.

\begin{definition}
Let $(X,0)$ be a space curve germ in $(\mathbb{C}^{n},0)$ with isolated singularity. If $(X,0)$ is reduced, then the \textit{Milnor number} is defined as in  \rm{\cite{buch}} \textit{(see Section} \ref{unfoldingsection}\textit{). If $(X,0)$ is not reduced, but it is generically reduced, i.e. $(X,0)$ can be written as $X=X_{\mathrm{red}}\cup E$, where $X_{\mathrm{red}}$ is the reduced curve germ given by the $1$-dimensional part of $X$, and $E$ is the $0$-dimensional part which contains the embedded point of $X$ at the origin, then we define the \textit{Milnor number} as
\[
\mu(X,0)=\mu(X_{\mathrm{red}},0)-2\epsilon(X,0),
\]
where $\epsilon(X,0)=\dim_{\mathbb{C}}N$ and $N$ is the nilradical of $\mathcal{O}_{X}$. If $(X,0)$ is defined by the ideal $I_{(X,0)}\subset\mathcal{O}_n$, then the $\epsilon$-invariant of $X$ at $0$ can be calculated by}
\begin{center}
$\epsilon(X,0)=\dim_\mathbb{C}\dfrac{\sqrt{I_{(X,0)}\mathcal{O}_{(X,0)}}}{I_{(X,0)}\mathcal{O}_{(X,0)}}.$   
\end{center}

\end{definition}
In the case that $(X,0)$ is smoothable (which will be always our case), the Milnor number can be seen as the number which measures the \textit{vanishing cohomology} of $X$. More exactly, let $\pi:\mathcal{X}\longrightarrow D$ be a good representative of a smoothing of $X$. Then, for $t\in D\setminus\{0\}$,
\[
\mu(X)=1-\chi(X_{t}),
\]
where $X_{t}=\pi^{-1}(t)$ and $\chi(X_{t})$ is the Euler characteristic of $X_{t}$.

Let $f:(\mathbb{C}^{n},0)\longrightarrow(\mathbb{C}^{2n-1},0)$ be a finitely determined map germ. Then the double point curves
\[
D^{2}(f),\quad D^{2}(f)/S_{2},\quad D(f)_{\mathrm{red}},\quad f(D(f))_{\mathrm{red}}
\]
are reduced and we can consider their Milnor numbers,
\[
\mu(D^{2}(f)),\quad
\mu(D^{2}(f)/S_{2}),\quad
\mu(D(f)_{\mathrm{red}}),\quad
\mu(f(D(f))_{\mathrm{red}}),
\]
which in fact give $\mathcal{A}$-invariants of $f$ (see \cite{victor-juanjo}).

\begin{theorem}\label{TEO2} Let $F$ be a $1$-parameter stabilization of a finitely determined map germ $f \colon (\mathbb{C}^n,0) \longrightarrow (\mathbb{C}^{2n-1},0)$. Then\\

\noindent(a) $D_F(f)$ and $f_F(D(f))$ are germs of curves with isolated singularity. Furthermore, the families of curves $\pi \colon D(F)_{\mathrm{red}} \longrightarrow (\mathbb{C},0)$ and $\pi' \colon F(D(F))_{\mathrm{red}} \longrightarrow (\mathbb{C},0)$ are flat deformations of $D_F(f)$ and $f_F(D(f))$, respectively, with generic fibres $\pi^{-1}(t)=D(f_t)$, $(\pi')^{-1}(t)=f_t(D(f_t))$ for $t\neq 0$.\\

\noindent(b) The invariants $\epsilon(D_F(f))$, $\epsilon(f_F(D(f)))$, $\mu(D_F(f))$ and $\mu(f_F(D(f)))$ do not depend on the choice of the stabilization $F$.\\

\noindent(c) Although the analytic spaces $D_F(f)$ and $f_F(D(f))$ may depend
on the choice of the stabilization $F$, by part (b) their numerical
invariants $\epsilon(D_F(f)), \mu(D_F(f)), \epsilon(f_F(D(f))),$ and $\mu(f_F(D(f)))$ do not. Henceforth, only at the level of these numerical invariants, we omit the subscript $F$ and write
\[
\mu(D(f),0):=\mu(D_F(f),0),
\qquad
\epsilon(D(f),0):=\epsilon(D_F(f),0),
\]
and
\[
\mu(f(D(f)),0):=\mu(f_F(D(f)),0),
\qquad
\epsilon(f(D(f)),0):=\epsilon(f_F(D(f)),0).
\]
With this convention,
\begin{center}
 $\mu(D(f),0)=\mu(D^2(f),0)$ $ \ \ $  and $ \ \ $   $\mu(D(f))=2\mu(f(D(f)))+C(f)-1$.   
\end{center}
\end{theorem}

\begin{proof} First, let us prove (a). Since $F$ is a stabilization, $D^2(F)$ is pure of
dimension $2$. The projection
\[
p_1:D^2(F)\longrightarrow D(F)_{\mathrm{red}}
\]
given by $p_1(x,x',t)=(x,t)$ is finite and surjective onto its image. Hence
$D(F)_{\mathrm{red}}$ is a reduced complex analytic space of pure
dimension $2$. By definition, $D_F(f)$ is the central fiber of the
projection
\[
\pi:D(F)_{\mathrm{red}}\longrightarrow (\mathbb{C},0),
\]
and its support is $D(f)_{\mathrm{red}}$, which has dimension $1$.
Therefore, by Lemma \ref{lemaaux5}, the projection $\pi$ is flat.

Similarly, the composition
\[
F\circ p_1:D^2(F)\longrightarrow F(D(F))_{\mathrm{red}}
\]
is finite and surjective onto its image. Thus
$F(D(F))_{\mathrm{red}}$ is also reduced and pure of dimension $2$.
By definition, $f_F(D(f))$ is the central fiber of
\[
\pi':F(D(F))_{\mathrm{red}}\longrightarrow(\mathbb{C},0),
\]
and its support is $f(D(f))_{\mathrm{red}}$, which has dimension $1$.
Hence, again by  Lemma \ref{lemaaux5}, $\pi'$ is flat.

Finally, since $f$ is finitely determined, it is stable away from the
origin. Consequently, $D(f)_{\mathrm{red}}$ and
$f(D(f))_{\mathrm{red}}$ are reduced and smooth away from the origin.
Thus $D_F(f)$ and $f_F(D(f))$ are germs of curves with isolated
singularity. Therefore,
\[
\pi:D(F)_{\mathrm{red}}\longrightarrow(\mathbb{C},0)
\qquad\text{and}\qquad
\pi':F(D(F))_{\mathrm{red}}\longrightarrow(\mathbb{C},0)
\]
are flat deformations of $D_F(f)$ and $f_F(D(f))$, respectively.


To prove (b), note that $D(F)_{\text{red}}$ is reduced; therefore, $\pi^{-1}(t) := D(f_t)$ is reduced for all $t \neq 0$ (see \cite[Prop. 7.7.(1).iv]{Greuelnovo}). Since $F$ is a stabilization of $f$, $D(f_t)$ is smooth for all $t \neq 0$ and therefore $\mu(D(f_t))=0$ (see \cite[Corollary 2.4(1)]{victor-juanjo}).

We will now prove that the invariants $\mu(D_F(f))$ and $\epsilon(D_F(f))$ do not depend on the choice of $F$. Let $F=(f_t,t)$ be a stabilization of $f.$ By part (a), the projection
\[
\pi:D(F)_{\mathrm{red}}\longrightarrow(\mathbb{C},0)
\]
is a flat deformation of $D_F(f)$. On the other hand,
\[
\pi:D^2(F)\longrightarrow(\mathbb{C},0)
\]
is the usual unfolding of $D^2(f)$ associated with the stabilization $F$.

Fix a representative $F:U\times T\rightarrow \mathbb{C}^{2n-1}\times T$, where $U$ and $T$ are neighborhoods of the origin in $\mathbb{C}^n$ and $\mathbb{C},$ respectively, such that $\pi:D^2(F)\rightarrow T$ and $\pi: D(F)_{\mathrm{red}}\rightarrow T$ are good representatives, in the sense of \cite{Brucker} and \cite{buch}. Thus, by \cite[Th. 4.2.2]{buch} and \cite[Lemma 3.1.2]{Brucker}, we have\begin{equation}\label{D^2b} 
        \mu(D^2(f),0)-\sum_{(x,x')\in D^2(f_t)}\mu(D^2(f_t),(x,x'))=1-\chi(D^2(f_t))
    \end{equation} and \begin{equation}\label{Db}
        \mu(D_F(f),0)-\sum_{x\in D(f_t)}\mu(D(f_t),x)=1-\chi(D(f_t)),
    \end{equation} for all $t\in T,$ where $\chi(\textbf{Y})$ denotes the Euler characteristic of the complex space $\textbf{Y}.$ Since $F$ is a stabilization of $f,$ it follows that $D^2(f_t)$ and $D(f_t)$ are smooth for $t\neq 0$ (see \cite[Prop. 2.6 and Corollary 2.4(1)]{victor-juanjo}) and $D(f_t)$ is reduced for $t\neq 0$. Thus $$\sum_{(x,x')\in D^2(f_t)}\mu(D^2(f_t),(x,x'))=\sum_{x\in D(f_t)}\mu(D(f_t),x)=0.$$
    
    By \cite[Cor. 2.4]{victor-juanjo}, we obtain that the projection $p:D^2(f_t)\rightarrow D(f_t)$ is a biholomorphism, for all sufficiently small $t\ne0$. Therefore, $\chi(D^2(f_t))=\chi(D(f_t))$ for $t\neq 0$. Thus, by \eqref{D^2b} and \eqref{Db}, we conclude that
    \begin{equation}\label{eq31}
   \mu(D^2(f),0)=\mu(D_F(f),0).      
    \end{equation}
     
    Now, if $G$ is another stabilization of $f$ we can do the same process to conclude that $\mu(D^2(f),0)=\mu(D_G(f),0).$ Since $D^2(f)$ is reduced, then $\mu(D_F(f),0)=\mu(D_G(f),0)$. Now, using that
    \begin{center}
    $\mu(D_F(f),0)=\mu(D(f)_{\mathrm{red}},0)-2\cdot\epsilon(D_F(f),0)$\ \ and\ \
    $\mu(D_G(f),0)=\mu(D(f)_{\mathrm{red}},0)-2\cdot\epsilon(D_G(f),0)$,     
    \end{center}
    
  \noindent  we obtain that $\epsilon(D_F(f),0)=\epsilon(D_G(f),0)$. This shows that $\mu(D_F(f),0)$ and $\epsilon(D_F(f),0)$ do not depend on the choice of the stabilization $F$. Furthermore, this also shows that the first formula in (c) follows by (\ref{eq31}).\\

  It remains to show that $\mu(f_F(D(f)))$ and $\epsilon(f_F(D(f)))$ do not depend on $F$, as well as the second formula in (b). To do this, consider again a stabilization \(F=(f_t,t)\) of \(f\), and fix \(t\neq 0\) sufficiently small. Set $D_t=D(f_t)$
and $Y_t=f_t(D(f_t))$. By (a), the families \(D(F)_{\mathrm{red}}\to T\) and \(F(D(F))_{\mathrm{red}}\to T\) are flat deformations of \(D_F(f)\) and \(f_F(D(f))\), respectively, and their generic fibres are \(D_t\) and \(Y_t\). Since \(F\) is a stabilization, \(D_t\) is smooth. We first prove the following statement.\\

\noindent \textbf{Statement:} $\chi(D_t)=2\chi(Y_t)-C(f).$\\

\noindent \textit{Proof of Statement.} Indeed, let \(A_t\subset D_t\) be the set of cross-caps of \(f_t\), and put
\(B_t=f_t(A_t)\). Since \(f_t\) is stable, $\#A_t=\#B_t=C(f).$ Moreover, since \(n\geq3\), a stable map $f_t:\mathbb C^n\longrightarrow\mathbb C^{2n-1}$ has no ordinary triple points. Hence, away from the cross-caps, every
point of \(Y_t\) has exactly two distinct preimages in \(D_t\), and the
restriction

$$
f_t:D_t\setminus A_t\longrightarrow Y_t\setminus B_t
$$

\noindent is an unramified two-sheeted covering.

For completeness, let us also describe the behaviour at a cross-cap.
Up to local biholomorphic changes of coordinates, \(f_t\) has the normal
form

$$
(u_1,\ldots,u_{n-1},y)
\longmapsto
(u_1,\ldots,u_{n-1},y^2,u_1y,\ldots,u_{n-1}y).
$$

Its double point curve is locally given by $u_1=\cdots=u_{n-1}=0,$ and the restriction of \(f_t\) to this curve is equivalent to $y\longmapsto y^2.$ Thus \(D_t\) and \(Y_t\) are smooth at the corresponding points and each
cross-cap gives precisely one simple branch point of
\(D_t\to Y_t\). It follows from multiplicativity of the Euler characteristic for finite coverings that $\chi(D_t\setminus A_t)=2\chi(Y_t\setminus B_t).$ Using additivity of the Euler characteristic and
\(\#A_t=\#B_t=C(f)\), we obtain

$$
\chi(D_t)-C(f)
=
2\bigl(\chi(Y_t)-C(f)\bigr),
$$

\noindent which proves the Statement.

Now \(D_t\) and \(Y_t\) are the generic fibres of the smoothings defining
the analytic structures on \(D_F(f)\) and \(f_F(D(f))\), respectively.
Therefore, $\mu(D_F(f),0)=1-\chi(D_t)$ and  $\mu(f_F(D(f)),0)=1-\chi(Y_t)$. Using the Statement, we obtain

$$
\begin{aligned}
\mu(D_F(f),0)
 &=1-\chi(D_t)\\
 &=1-\bigl(2\chi(Y_t)-C(f)\bigr)\\
 &=2\bigl(1-\chi(Y_t)\bigr)+C(f)-1\\
 &=2\mu(f_F(D(f)),0)+C(f)-1.
\end{aligned}
$$

\noindent Therefore, we obtain that 
\begin{equation}\label{eq32}
\mu(f_F(D(f)),0)=\dfrac{1}{2}\biggl(  \mu(D_F(f),0)-C(f)+1 \biggr).     
\end{equation}
We have already shown above that $\mu(D_F(f),0)$ does not depend on the choice of $F$. Furthermore, it is a well-known fact in the literature that $C(f)$ also does not depend on the choice of the stabilization. Therefore, it follows by (\ref{eq32}) that $\mu(f_F(D(f)),0)$ does not depend on $F$. Using that $\mu(f_F(D(f)),0)=\mu(f(D(f))_{\mathrm{red}},0)-2\cdot\epsilon(f_F(D(f)),0)$ we also obtain that $\epsilon(f_F(D(f)),0)$ does not depend on the choice of $F$, which completes the proof of (b). Furthermore, this also shows that the second formula in (c) follows by (\ref{eq32}).\end{proof}

\begin{remark} Let \(f:(\mathbb{C}^2,0)\to(\mathbb{C}^3,0)\) be a finitely determined map germ, and let \(F=(f_t,t)\) be a $1$-parameter stabilization of \(f\). Consider the analytic spaces \(D(F)\) and \(F(D(F))\), defined via Fitting ideals as in Definition \rm{\ref{DefD}}\textit{. In this setting, both \(D(F)\) and \(F(D(F))\) are Cohen--Macaulay spaces; see} \cite{mond-pellikaan}\textit{. If we define \(D_F(f)\) and \(f_F(D(f))\) in the same way as in the case \(n\geq 3\), then, in this setting, these spaces are reduced curves; see, for instance,} \cite[\textit{Section} 10.1]{Greuelnovo}\textit{. This is precisely the analytic structure commonly used in the literature for finitely determined map germs from \((\mathbb{C}^2,0)\) to \((\mathbb{C}^3,0)\). Thus, our construction may also be viewed as a generalization of the classical case \(n=2\).}
\end{remark}

\section*{Acknowledgments}	
$ \ \ \ \ $ The authors would like to thank the anonymous referees for their careful reading of the manuscript and for their valuable comments and suggestions, which substantially improved the presentation and clarity of the paper. This work forms part of the second author's Ph.D. thesis at the Federal University of Paraíba, under the supervision of Otoniel Nogueira da Silva, to whom he would like to express his deepest gratitude. The first author was supported by the Conselho Nacional de Desenvolvimento Científico e Tecnológico (CNPq) under Grants Universal 407454/2023-3 and 400043/2025-4. The second author was supported by the Coordenação de Aperfeiçoamento de Pessoal de Nível Superior (CAPES).

\end{document}